\documentclass[12pt]{article}
\usepackage{geometry}                		% See geometry.pdf to learn the layout options. There are lots.
\usepackage[parfill]{parskip}    		% Activate to begin paragraphs with an empty line rather than an indent
\usepackage{amssymb}
\usepackage{amsmath}

\usepackage{palatino}
\usepackage{textcomp}
\usepackage{graphicx}
\usepackage{epsfig}
\usepackage{epstopdf}
\usepackage{mathrsfs}
\usepackage[english]{babel}
\usepackage{amscd}
\usepackage{amsthm}

\usepackage{fancyhdr}
\newtheorem{thm}{Theorem}[section]
\newtheorem{lemma}[thm]{Lemma}
\newtheorem{prop}[thm]{Proposition}

\newtheorem{remark}[thm]{Remark}

\title{Hoeffding's inequality for sums of weakly dependent random variables}

\author{Christos Pelekis\thanks{Department of Computer Science, KU Leuven
Celestijnenlaan 200A,
3001, Belgium, Email: pelekis.chr@gmail.com}, \  \
Jan Ramon\thanks{Department of Computer Science KU Leuven
Celestijnenlaan 200A, 
3001, Belgium, Email: Jan.Ramon@cs.kuleuven.be} }

\begin{document}

\maketitle

%\section{}
%\subsection{}

%\renewcommand{\thefootnote}{\fnsymbol{footnote}} 
%\footnotetext{
%\emph{Keywords:} Hoeffding's inequality, sum of weakly dependent random variables, martingale differences, $k$-wise 
%independent random variables, $U$-statistics}     
%\renewcommand{\thefootnote}{\arabic{footnote}}

\begin{abstract}  
We provide a systematic approach to deal with the following problem. Let $X_1,\ldots,X_n$ be,  
possibly dependent, $[0,1]$-valued random variables. 
What is a sharp upper bound on the probability 
that their sum is significantly larger than their mean? 
In the case of independent random variables, a fundamental tool for bounding such 
probabilities is devised by Wassily Hoeffding. In this paper we consider analogues of Hoeffding's 
result for sums of dependent random variables for which we have certain information on their 
dependency structure. We 
prove a result that 
yields concentration 
inequalities for several notions of weak  dependence between random variables.  
Additionally,   
we obtain a new concentration inequality for sums of, possibly dependent, $[0,1]$-valued 
random variables, $X_1,\ldots,X_n$, that satisfy the 
following condition: there exist constants $\gamma \in (0,1)$ and $\delta\in (0,1]$ such that 
for every subset $A\subseteq \{1,\ldots,n\}$ we have
$\mathbb{E}\left[\prod_{i\in A} X_i \prod_{i\notin A}(1-X_i) \right]\leq \gamma^{|A|} \delta^{n-|A|}$,
where $|A|$ denotes the cardinality of $A$. Our approach applies to  
several sums of weakly dependent random variables such as 
sums of martingale difference sequences, sums of $k$-wise independent random variables and $U$-statistics. 
Finally, we discuss some applications to the theory of random graphs. 
\end{abstract}


\section{Prologue, related work and main results}\label{prologue}

\subsection{Covariance estimates}

The main purpose of this work is to 
obtain extensions of Hoeffding's inequality to  sums of weakly dependent random variables. 
In order to emphasize the analogy between existing and posterior results,
let us begin right away 
by stating Hoeffding's well-known theorem (see \cite{Hoeffdingone}, Theorem $1$).  
Throughout the text, $\mathbb{E}[\cdot]$ and 
$\mathbb{P}[\cdot]$ 
will denote expectation and probability, respectively.    \\

\begin{thm}[Hoeffding, $1963$]
\label{folklore}
Let  $X_1,\ldots,X_n$ be \emph{independent} random variables such that $0\leq X_i \leq 1$, for each 
$i=1,\ldots,n$.
Set $p = \frac{1}{n}\sum_{i=1}^{n}\mathbb{E}[X_i]$ and fix a real number $\varepsilon$
from the interval $\left(0,\frac{1}{p}-1\right)$. If $t=np+np\varepsilon$ then
\[ \mathbb{P}\left[\sum_{i=1}^{n}X_i  \geq t \right] \leq   \inf_{h>0}\; e^{-ht}\left(1-p + pe^h\right)^n . \]
Furthermore, 
\[  \inf_{h>0} \; e^{-ht}\left(1-p + pe^h\right)^n  =p^t (1-p)^{n-t} \left(\frac{n-t}{t}\right)^{t} \left(\frac{n}{n-t}\right)^{n} := H(n,p,t)  \]
and  
\[H(n,p,t) = e^{-n D(p+p\varepsilon||p)}, \]
where for $q,p\in (0,1)$, $D(q||p) = q\ln\frac{q}{p}+(1-q)\ln\frac{1-q}{1-p}$ is the Kullback-Leibler distance between $q$ and $p$.  
\end{thm}

The function $H(n,p,t)$ is the so-called \emph{Hoeffding function}. The estimate 
$e^{-n D(p+p\varepsilon||p)}$, i.e. the Hoeffding function expressed in terms of the Kullback-Leibler distance, 
is referred to as the \emph{Chernoff-Hoeffding bound}. 
In other words, Hoeffding's result provides an upper bound 
on the probability that a sum of \emph{independent and bounded} random variables is 
significantly larger than its expected value. 
We remark that a foolproof version of the bound can be obtained using the standard estimate on 
the Kullback-Leibler distance: $D(q||p) \geq 2(q-p)^2$, for $q,p\in (0,1)$ such that $q\geq p$.
Hoeffding's inequality is a folklore result that 
has been proven to be useful in a plethora of problems in combinatorics, probability, statistics and 
theoretical computer science. 
However, there are several instances in which 
one is dealing with    
sums of bounded random variables that are \emph{not} independent;  as an example the reader may 
think of the number of triangles in an Erd\H{o}s-R\'enyi random graph.  
Such instances have been encountered by several authors for a variety of questions
which, in succession, gave rise to the problem of obtaining 
analogues of Theorem \ref{folklore} for sums of dependent random 
variables, under certain assumptions on their dependency-structure. The amount of literature 
that treats the problem of extending Hoeffding's theorem to 
sums of dependent random variables is vast and the interested 
reader is invited to take a look at 
the works of Azuma \cite{Azuma}, Bentkus \cite{Bentkusone}, Delyon \cite{Delyon},
Fan et al. \cite{FanGrama}, Gavinsky et al. \cite{Gavinsky}, 
Gradwohl et al. \cite{Gradwohl}, Hazla et al. \cite{Hazla}, Impagliazzo et al. 
\cite{Impagliazzo}, Janson \cite{Janson}, Kallabis et al.  \cite{Kallabis},
Kontorovich et al. \cite{Kontorovich}, Linial et al. \cite{Linial},
 McDiarmid \cite{McDiarmid}, Ramon et al. \cite{Ramon}, Rio \cite{Rio}, Schmidt et al. 
\cite{Schmidt}, Siegel \cite{Siegelone}, Van de Geer \cite{Geer}, Vu \cite{Vu}, among others. 
Let us also remark that certain assumptions of "weak dependence" between the random variables 
are required in order to make the problem interesting. If the random  
variables are fully dependent then the problem is trivial; just let $X_1=\cdots =X_n =\frac{t}{n}$ with probability 
$\frac{1}{t}\sum_i\mathbb{E}[X_i]$ and   $X_1=\cdots =X_n =0$ with 
probability $1-\frac{1}{t}\sum_i\mathbb{E}[X_i]$. Then, for $t \geq \sum_i\mathbb{E}[X_i]$, Markov's 
inequality implies that  
$\mathbb{P}\left[\sum_i X_i \geq t\right]\leq \frac{\sum_i\mathbb{E}[X_i]}{t}$ and the later collection 
of random variables attains this bound. 
This article may be regarded as an addendum to the aforementioned amount of literature; 
we prove a result that can be employed in order to obtain    
concentration inequalities for sums of \emph{dependent} 
random variables for which we have certain information on their dependency structure.  \\
The exposition of our paper proceeds as follows. In the remaining part of the current section we formalise 
a particular type of "dependency-structure" between bounded random variables and 
juxtapose existing bounds on the probability that their sum is larger than their mean with 
bounds obtained via our approach.  
There are several ways to describe a dependency structure between random variables, some 
of which will be discussed in the following subsections. 
Let us begin with a rather general 
description that
assumes estimates on the "covariance structure" of the random variables and is contained in 
the following theorem, due to Impagliazzo and Kabanets \cite{Impagliazzo}. 
Here and later, for a positive integer $n$, we will denote by $[n]$ the set $\{1,\ldots,n\}$.\\

\begin{thm}[Impagliazzo \& Kabanets, $2010$]
\label{Impaglia} There exists a universal constant $c\geq 1$ satisfying the following.
Suppose that $X_1,\ldots,X_n$ are random variables such that $0\leq X_i\leq 1$, for $i=1,\ldots,n$. 
Assume further that 
there exists constant $\gamma\in (0,1)$ such that for all 
$A\subseteq [n]$ the following condition holds true:
\[ \mathbb{E}\left[ \prod_{i\in A}X_i\right] \leq \gamma^{|A|}, \]
where $|A|$ denotes the cardinality of $A$. 
Fix a real number $\varepsilon$ from the interval $\left(0, \frac{1}{\gamma}-1\right)$
and set $t = n\gamma + n\gamma \varepsilon$. Then
\[ \mathbb{P}\left[\sum_{i=1}^{n}X_i \geq  t\right] \leq c e^{-nD(\gamma(1+\varepsilon) || \gamma)},\]
where  $D(\gamma(1+\varepsilon) || \gamma)$
is the Kullback-Leibler distance between 
$\gamma(1+\varepsilon)$ and $\gamma$.
\end{thm}

Throughout the text, the empty product is interpreted as $1$. 
See \cite{Impagliazzo} for a neat proof of the previous 
result as well as for applications to direct products and expander graphs, among others.  
In the case of Bernoulli $0/1$ random variables it is shown in \cite{Impagliazzo}, Theorem $3.1$, that 
the constant $c$ in the previous theorem is equal to $1$;  however 
the exact value of $c$ 
does \emph{not} seem to be known in the case of general $[0,1]$-valued random variables. 
Moreover, in the case of Bernoulli $0/1$ random variables, the following refinement upon 
Theorem \ref{Impaglia} has been obtained by
Linial and Luria \cite{Linial}.\\

\begin{thm}[Linial \& Luria, $2014$]
\label{Linial} Let $X_1,\ldots,X_n$ be Bernoulli $0/1$ random variables. Let $\beta\in (0,1)$ 
be such that $\beta n$ is a positive integer and let $k$ be 
any positive integer such that $0< k < \beta n$. Then 
\[ \mathbb{P}\left[\sum_{i=1}^{n}X_i\geq \beta n\right] \leq \frac{1}{\binom{\beta n}{k}} \sum_{A: |A|=k} \mathbb{E}\left[\prod_{i\in A}X_i\right] . \]
\end{thm}

See \cite{Linial} for a very elementary proof of this result. Notice that the previous result reduces 
to Markov's inequality when $k=1$. 
It can be seen, using standard entropy estimates of binomials, that  Theorem \ref{Linial}  
reduces to Theorem \ref{Impaglia} in case one makes the 
additional assumtion $\mathbb{E}\left[\prod_{i\in A}X_i\right]\leq
\gamma^{|A|}$, for all $A\subseteq [n]$. 
We provide two proofs of Theorem \ref{Linial} in Section \ref{mainres}. The first proof   
is based upon the main result of our paper which provides a concentration bound, for sums 
of random variables,  
expressed in terms 
of expectations with respect to convex functions. 
More precisely, a basic ingredient in the proof of  
most results in this paper is the following theorem. 
Here and later, we will denote by $\partial_j[n]$ the family consisting of all 
subsets of $[n]$ whose cardinality equals $j \in \{0,1\ldots,n\}$.\\

\begin{thm}\label{depHoeff}  
Let $X_1,\ldots, X_n$ be random variables such that $0\leq X_i\leq 1$, 
for $i=1,\ldots,n$.  
For every subset $A\subseteq [n]$, define the random variable, $Z_A$, by setting
\[ Z_A = \prod_{i\in A} X_i \prod_{i\in [n]\setminus A} (1-X_i) .\]
Let $\mathcal{F}$ be the set consisting of all functions $f:\mathbb{R} \rightarrow [0,+\infty)$ 
that are \emph{increasing} and \emph{convex} and
set $p:=\frac{1}{n}\sum_{i=1}^{n}\mathbb{E}[X_i]$. If $t$ is a real number 
such that $np<t<n$, 
then $\sum_A \mathbb{E}\left[Z_A\right] =1$ and
\[ \mathbb{P}\left[\sum_{i=1}^{n}X_i \geq t\right] \leq \inf_{f\in \mathcal{F}} \frac{1}{f(t)}\; \mathbb{E}\left[f(Z)\right],  \]
where $Z$ is the random variable that takes values in the set $\{0,1,\ldots,n\}$ with probability 
\[ \mathbb{P}[Z=j]=\sum_{A\in \partial_j[n]}\mathbb{E}\big[ Z_A\big], \; \text{for}\; j=0,1,\ldots,n . \]
\end{thm}

Let us remark that the assumption $t>np$, in the previous theorem, is essential.  
Indeed, a first step in proof of the previous theorem is an application of Markov's inequality:
\[\mathbb{P}\left[\sum_{i=1}^{n}X_i \geq t\right] \leq \frac{1}{f(t)}\; \mathbb{E}\left[f\left(\sum_{i=1}^{n}X_i\right)\right] ,\; \text{for}\; f \in \mathcal{F} . \]
Since $f$ is assumed to 
be convex, Jensen's inequality implies $\mathbb{E}\left[f(\sum_{i=1}^{n}X_i)\right] \geq f(np)$.
Since $f$ is additionally assumed to be increasing we have $f(t)\leq f(np)$,  
for $t\leq np$, and so the aforementioned application of Markov's inequality 
cannot yield a useful estimate. If $t<np$, then the previous theorem gives a useful upper bound on the 
probability $\mathbb{P}\left[\sum_{i=1}^{n}X_i \leq t\right] = \mathbb{P}\left[n-\sum_{i=1}^{n}X_i \geq n-t\right]$; 
so we may choose to work with the upper tail.  
We prove Theorem \ref{depHoeff}  in 
Section \ref{mainres}. The proof makes use of an elementary result  
(Lemma \ref{mainlem} below) that allows one to write a sum of $n$ real 
numbers from the interval $[0,1]$ as 
a convex combination of the set of integers $\{0,1,\ldots,n\}$. We also show that, in the case of \emph{independent} random variables,  Theorem \ref{depHoeff} 
reduces to Hoeffding's Theorem \ref{folklore}. It turns out that Theorem \ref{depHoeff} can be 
employed in order to obtain concentration inequalities for  
several sums of weakly dependent random variables such as 
martingale difference sequences, $k$-wise independent random variables and sums of 
Bernoulli $0/1$ random variables whose dependency structure is given in terms of a graph.
We illustrate this in the following subsections. 
Let us begin with a consequence of Theorem \ref{depHoeff} that may be seen as an 
generalisation of Hoeffding's Theorem \ref{folklore}. \\

\begin{thm}\label{expfunct} 
Suppose that $X_1,\ldots,X_n$ are random variables such that $0\leq X_i\leq 1$, for $i=1,\ldots,n$. 
Assume further that 
there exist constants $\gamma\in (0,1)$ and $\delta \in (0,1]$ such that for all 
$A\subseteq [n]$ the following condition holds true:
\[ \mathbb{E}\left[Z_A \right] \leq \gamma^{|A|}\cdot \delta^{n-|A|}, \; \text{where}\; Z_A =  \prod_{i\in A}X_i \prod_{i\in [n]\setminus A}(1-X_i) \]
and $|A|$ denotes the cardinality of $A$. Fix a real number $\varepsilon$ 
from the interval $\left(0, \frac{1}{\gamma}-1\right)$
and set $t = n\gamma + n\gamma \varepsilon$.
Then
\[ \mathbb{P}\left[\sum_{i=1}^{n}X_i \geq t\right] \leq \gamma^t \delta^{n-t} \left(\frac{n-t}{t}\right)^{t} \left(\frac{n}{n-t}\right)^{n} . \]
Furthermore, 
\[ \gamma^t \delta^{n-t} \left(\frac{n-t}{t}\right)^{t} \left(\frac{n}{n-t}\right)^{n}  \leq  e^{-n\left\{D(\gamma(1+\varepsilon) || \gamma) - \left(1-\gamma(1+\varepsilon)\right) \ln \frac{\delta}{1-\gamma} \right\}} , \]
where $D(\gamma(1+\varepsilon) || \gamma)$ denotes the Kullback-Leibler distance between 
$\gamma(1+\varepsilon)$ and $\gamma$.  
\end{thm}

We prove this result in Section \ref{mainres}. 
In other words,  the previous result adjusts the factor 
$\left(1-\gamma(1+\varepsilon)\right) \ln \frac{\delta}{1-\gamma}$ to the Chernoff-Hoeffding bound in 
retaliation for the fact that the random variables were not assumed to be independent. 
We remark that we always have $\gamma + \delta\geq 1$. To see this notice that, 
since $\sum_A \mathbb{E}\left[Z_A\right] =1$, the condition of the previous theorem implies  
\[ 1 = \sum_{j=0}^{n} \sum_{A\in\partial_j[n]} \mathbb{E}\left[Z_A\right] \leq \sum_{j=0}^{n}\binom{n}{j}\gamma^j \delta^{n-j}  = \left(\gamma + \delta \right)^n.\] 
Notice also that the 
factor $\left(1-\gamma(1+\varepsilon)\right) \ln \frac{\delta}{1-\gamma}$ is not very large  
(for example, it less than $(1-\gamma)\ln\frac{1}{1-\gamma}\leq \frac{1}{e}$) and that 
the bound of Theorem \ref{expfunct} involves no unknown constants. \\

\begin{remark} Theorem \ref{expfunct} should be considered as complementary to Theorem \ref{Impaglia}, 
in the sense that it may be applicable when an estimate of the form   
$\mathbb{E}\left[ \prod_{i\in A}X_i\right] \leq \gamma^{|A|}$ is not available and, instead, an estimate of   
the form $\mathbb{E}\left[Z_A \right] \leq \gamma^{|A|}\cdot \delta^{n-|A|}$ is available. 
Let us also remark that an estimate of the former form cannot be concluded from an estimate of the
later form and so 
one cannot conclude Theorem \ref{expfunct} 
as a consequence of Theorem \ref{Impaglia}. 
To be more precise, let us look at the case of Bernoulli $0/1$ random variables. 
In that case the constant $c$ in Theorem \ref{Impaglia} equals $1$ (see \cite{Impagliazzo}, 
Theorem $3.1$). 
Under the assumption   
$\mathbb{E}\left[Z_A \right] \leq \gamma^{|A|}\cdot \delta^{n-|A|}$, for all $A\subseteq [n]$, and 
since the random variables are Bernoulli we have
\begin{eqnarray*} \mathbb{E}\left[ \prod_{i\in A}X_i\right] &=& \sum_{T: A\subseteq T} 
\mathbb{E}\left[Z_T\right] \\
&\leq&\sum_{j=0}^{n-|A|}\binom{n-|A|}{j} \gamma^{|A|+j} \delta^{n-|A|-j} = \gamma^{|A|} \left(\gamma +\delta\right)^{n-|A|}
\end{eqnarray*}
and so Theorem \ref{expfunct} is dealing with an estimate   
on $\mathbb{E}\left[ \prod_{i\in A}X_i\right]$ that is, for fixed $\gamma$,  larger than the corresponding estimate in Theorem 
\ref{Impaglia}. 
\end{remark}

The second proof of Theorem \ref{Linial} is obtained using a coupling argument.
In fact, we prove a bit more. \\

\begin{thm}
\label{Linialrefined} Let $X_1,\ldots,X_n$ be Bernoulli $0/1$ random variables. Let $\beta\in (0,1)$ 
be such that $\beta n$ is a positive integer and let $k$ be 
any positive integer such that $0< k < \beta n$. Then  
\[  \frac{1}{\binom{n}{\beta n}} \sum_{A: |A|=\beta n} \mathbb{E}\left[\prod_{i\in A}X_i\right] \leq  \mathbb{P}\left[\sum_{i=1}^{n}X_i\geq \beta n\right] \leq \frac{1}{\binom{\beta n}{k}} \sum_{A: |A|=k} \mathbb{E}\left[\prod_{i\in A}X_i\right] . \]
\end{thm}

In Section \ref{mainres}  
we provide two proofs of the upper bound. 
Let us remark 
that the second proof is basically a paraphrasis, in probabilistic language, of the combinatorial proof from \cite{Linial}.
Our proof is longer and rather uglier but reveals a way to think of  
a lower bound.
Let us also remark that Theorem \ref{Impaglia} and Theorem \ref{Linial}   
may be employed in order to obtain concentration 
bounds for particular sums of dependent indicators that are encountered in the theory of 
Erd\H{o}s-R\'enyi random graphs; we illustrate this in Section \ref{isolated}. 
Moreover, we obtain the following result that is related to Theorem \ref{Impaglia}. \\

\begin{thm}\label{bincoupling} 
Suppose that $X_1,\ldots,X_n$ are random variables such that $0\leq X_i\leq 1$, for $i=1,\ldots,n$. 
Set $p =\frac{1}{n}\sum_i \mathbb{E}\left[X_i\right]$ and fix a real number $t$ such that 
$np + 1 < t < n$.
If $\varepsilon_0>0$ is such that $t-1=np+np\varepsilon_0$, then 
\[ \mathbb{P}\left[\sum_{i=1}^{n}X_i \geq  t\right] \leq 2 e^{-nD(p(1+\varepsilon_0) || p)},\]
where  $D(p(1+\varepsilon_0) || p)$
is the Kullback-Leibler distance between 
$p(1+\varepsilon_0)$ and $p$.
\end{thm}

We prove this result in Section \ref{mainres}. Notice that the constant $c$ 
of Theorem \ref{Impaglia} has been replaced by $2$, but the parameter $\varepsilon_0$ is smaller 
than the corresponding parameter $\varepsilon$ in Theorem \ref{Impaglia}, which results to a slightly 
larger exponential bound. \\
It turns out that Theorem \ref{depHoeff} applies 
to sums of martingale difference 
sequences. This is the content of the following subsection. 

\subsection{Martingales}\label{subsectionmartingale}

Martingales are sequences of  
random variables that exhibit a rather simple dependency structure. 
More precisely,  a sequence $X_0,X_1\ldots$ of integrable random variables is called a \emph{martingale} if 
\[ \mathbb{E}\big[X_{n+1}| \mathcal{A}_n\big] = X_n, \; \text{for all}\; n\geq 0 ,\]
where $\mathcal{A}_n$ is the $\sigma$-algebra generated by the random variables $X_0,X_1,\ldots,X_n$.
A sequence $Y_1,Y_2\ldots$ of integrable random variables is called a \emph{martingale difference sequence} 
if 
\[ \mathbb{E}\big[Y_n| \mathcal{F}_{n-1}\big]=0, \; \text{for all}\; n\geq 1 ,\]
where $\mathcal{F}_{n-1}$ is the $\sigma$-algebra generated by the random variables $Y_1,\ldots,Y_{n-1}$
and $\mathcal{F}_0$ is the trivial $\sigma$-algebra.  Given a martingale $X_0,X_1,\ldots$, one can 
obtain a martingale difference sequence by setting $Y_k = X_k -X_{k-1}, k=1,2,\ldots,$ and, 
conversely, given $X_0$ and a martingale difference sequence $Y_1,Y_2,\ldots,$ one can obtain a 
martingale by setting $X_k = X_0 + \sum_{i=1}^{k}Y_i$. Therefore, one may choose to work with either 
sequence. Theorem \ref{depHoeff} allows to prove a 
refined version of a 
well-known result, due to McDiarmid \cite{McDiarmid}, that
provides a 
concentration inequality for sums of martingale difference sequences. McDiarmid's 
inequality has been proven to 
be useful in several questions in combinatorics and probability and reads as follows.   \\

\begin{thm}[McDiarmid, $1989$]
\label{McDiar}
\label{genmcdm}  Let $Y_1,\ldots,Y_n$ be a martingale difference sequence with $-p_i\leq Y_i\leq 1-p_i$, for $i=1,\ldots,n$ 
and suitable constants $p_i\in (0,1)$. Set $p = \frac{1}{n}\sum_{i=1}^{n} p_i$. 
Then, for any real $t$ such that $t\in (0,1-p)$, we have 
\[  \mathbb{P}\left[\sum_{i=1}^{n} Y_i \geq nt  \right]   \leq \inf_{h>0} \; e^{-(hnt+hnp)} 
\mathbb{E}\left[e^{hB_{n,p}}\right] , \]
where  $B_{n,p}$ is a binomial random variable of parameters $n$ and $p$.
Furthermore, 
\[ \inf_{h>0} \; e^{-(hnt+hnp)} 
\mathbb{E}\left[e^{hB_{n,p}}\right] \leq \left\{ \left(\frac{p}{p+t}\right)^{p+t} 
\left(\frac{1-p}{1-p-t} \right)^{1-p-t} \right\}^n := H_m(n,p,t)  . \]
and the following foolproof version holds true:
\[ H_m(n,p,t)   \leq \text{exp}\left(-2nt^2\right)  .\]
\end{thm}

See McDiarmid \cite{McDiarmid}, Theorem $6.1$, for a proof of this result. 
Let us remark that the function $H_m(n,p,t)$ is related to the Hoeffding function; in fact, given $n,p,t$ as 
in Theorem \ref{McDiar}, we have 
$H_m(n,p,t) = H(n,p,nt)$.  
Using Theorem \ref{depHoeff} we deduce the following refined version of the previous result. \\

\begin{thm}
\label{genmcdm}  Let $Y_1,\ldots,Y_n$ be a martingale difference sequence with $-p_i\leq Y_i\leq 1-p_i$, for $i=1,\ldots,n$ 
and suitable constants $p_i\in (0,1)$. Set $p = \frac{1}{n}\sum_{i=1}^{n} p_i$. 
Let $\mathcal{F}$ the set consisting of all functions $f:\mathbb{R} \rightarrow [0,+\infty)$ 
that are \emph{increasing} and \emph{convex}.
Then, for any real $t$ such that $t\in (0,1-p)$, we have 
\[  \mathbb{P}\left[\sum_{i=1}^{n} Y_i \geq nt  \right]   \leq \inf_{f\in \mathcal{F}} \; \frac{1}{f(nt+np)} 
\mathbb{E}\left[f\left(B_{n,p}\right)\right] , \]
where  $B_{n,p}$ is a binomial random variable of parameters $n$ and $p$.
Furthermore, if $n(p+t)$ is a \emph{positive integer} and $t$ satisfies  
$\frac{p(1-p)(e-1)}{1-p+ep} <t < 1-p$, we have
\begin{eqnarray*} \inf_{f\in \mathcal{F}} \; \frac{1}{f(nt+np)} 
\mathbb{E}\left[f\left(B_{n,p}\right)\right]  &\leq& \frac{h+1}{e^{h}}\left\{ H_m(n,p,t) - T(n,p,t)\right\} \\
&+& \left(1-\frac{1+h}{e^h}\right) \mathbb{P}\left[B_{n,p}=n(p+t)\right],
\end{eqnarray*}
where 
\[T(n,p,t) :=\sum_{j\leq n(p+t)-1}e^{h(j-n(p+t))} \mathbb{P}\left[B_{n,p}=j\right], \]
$H_m(n,p,t)$ is the function defined in Theorem \ref{McDiar}
and $h$ is the positive real satisfying 
\[ e^{-hn(t+p)} \;
\mathbb{E}\left[e^{hB_{n,p}}\right]  = \inf_{s>0} \; e^{-(snt+snp)} \;
\mathbb{E}\left[e^{sB_{n,p}}\right] . \]
The bound is less than the bound of Theorem \ref{McDiar}.
\end{thm}

Let us prove the last statement of the previous result. To this end, notice that 
the bound of the second statement is
\[ \leq \frac{1+h}{e^{h}} H_m(n,p,t) + \left(1- \frac{1+h}{e^h}\right) \mathbb{P}\left[B_{n,p} =n(p+t)\right] := Q .\]
Now Hoeffding's Theorem \ref{folklore} implies that 
\[ \mathbb{P}\left[B_{n,p} =n(p+t)\right] \leq \mathbb{P}\left[B_{n,p} \geq n(p+t)\right] \leq H_m(n,p,t). \]
Since $Q$ is a convex combination of $\mathbb{P}\left[B_{n,p} =n(p+t)\right]$ and 
$H_m(n,p,t)$, it follows that the bound of the previous result is less than the bound of Theorem \ref{McDiar}. 
The proofs of the remaining statements of Theorem \ref{genmcdm} can be found in  
Section \ref{martsection}.
Our approach uses Theorem \ref{depHoeff} combined with extensions of ideas 
that we employed in previous work (see \cite{PeRaWa}).  
In the next subsection we apply Theorem \ref{depHoeff} to another class of weakly dependent random 
variables. 

\subsection{$k$-wise independence}

In this section we employ Theorem \ref{depHoeff} in order to 
obtain a concentration inequality for a particular class of weakly dependent random variables. 
We begin by first defining this notion of weak dependence. 
The random variables $X_1,\ldots,X_n$ will be called $k$-\emph{wise independent} if for any 
subset of $k$ indices $A = \{i_1,\ldots,i_k\}$ and all outcomes $x_{i_1},\ldots,x_{i_k}$ we have 
\[ \mathbb{P}\left[X_{i_1}\leq x_{i_1}\cap \cdots \cap X_{i_k}\leq x_{i_k} \right] = \prod_{i_j \in A} \mathbb{P}\left[X_{i_j}\leq x_{i_j}\right] . \]
$K$-wise independent random variables play a key role in theoretical computer science where 
they are used for de-randomizing algorithms (see \cite{Alon}). 
Note that $2$-independent random 
variables are just pairwise independent random variables. 
Let us also mention two examples of $(n-1)$-wise random variables. Let $G$ a graph on $n$ vertices.
Suppose that each edge of $G$ is given a random orientation with probability $1/2$ for each direction, 
independently of all other edges. For every $v\in G$, let $\delta_v = \text{deg}^{-}(v)\mod 2$, 
where $\text{deg}^{-}(v)$ is the in-degree of vertex $v$. Then (see \cite{PelSch}, Theorem $4$) the random variables $\delta_v, v\in V$ are $(n-1)$-wise independent. 
Similarly, let $G$ be a random graph from $\mathcal{G}(n,1/2)$ and for 
every vertex $v\in G$, set $d_v  = \text{deg}(v)\mod 2$. Then (see \cite{PelekisBern}, Corollary $4.2$) the random variables $\delta_v, v\in V$ are $(n-1)$-wise independent.
For more sophisticated examples on $k$-wise independent random variables 
we refer the reader to Alon et al. \cite{Alon} and 
Benjamini et al. \cite{Benjamini}. \\ 

We shall be interested in concentration inequalities for sums of $k$-wise independent random variables. 
The problem of obtaining analogues of Hoeffding's Theorem 
\ref{folklore} for sums of $k$-wise independent random variables has attracted the attention 
of several authors. See for example the works of 
Bellare et al. \cite{Bellare} and Schmidt et al. \cite{Schmidt} and references therein. 
Among the several existing concentration inequalities 
the following one is obtained via an approach that is similar to the approach of this paper. \\

\begin{thm}[Schmidt, Siegel, Srinivasan, $1995$]
\label{SSS} Let $X_1,\ldots,X_n$ be random variables such that $0\leq X_i\leq 1$ and $\mathbb{E}\left[X_i\right]= p_i$, for each $i=1,\ldots,n$. Set $p=\frac{1}{n}\sum_ip_i$.
Fix $\varepsilon >0$ and 
set $k_{\ast} := \lceil\frac{np\varepsilon}{1-p} \rceil$. If $k\geq k_{\ast}$ and 
$X_1,\ldots,X_n$ are $k$-wise independent then 
\[ \mathbb{P}\left[\sum_{i=1}^{n}X_i\geq np(1+\varepsilon)\right] \leq \binom{n}{k_{\ast}} p^{k_{\ast}} / \binom{np(1+\varepsilon)}{k_{\ast}} . \]
\end{thm}

See \cite{Schmidt} for a proof of this result, a basic ingredient of which is the use 
of elementary symmetric functions defined as 
$S_j(X_1,\ldots,X_n)= \sum_{A\in  \partial_j[n]} \prod_{i\in A} X_i$, where $j<n$.  
Clearly, the expectation of the function $S_k(X_1,\ldots,X_n)$ is related to the definition of the
random variable $Z$ in Theorem \ref{depHoeff}. In particular, the later result 
yields the following analogue of 
Hoeffding's Theorem \ref{folklore} for sums of $k$-wise independent random variables. \\

\begin{thm}\label{kwiseindep} Let $X_1,\ldots,X_n$ be $k$-wise 
independent random variables such that $0\leq X_i\leq 1$ and $\mathbb{E}\left[X_i\right]= p$, for each $i=1,\ldots,n$. Fix $\varepsilon >0$. Then
\[ \mathbb{P}\left[\sum_{i=1}^{n}X_i\geq np(1+\varepsilon)\right] \leq \frac{1}{(p-p^2)^{n-k}} e^{-nD(p(1+\varepsilon)||p)} . \]
\end{thm}

Notice that the previous result reduces to Hoeffding's Theorem \ref{folklore}, when $k=n$, i.e. the 
random variables are mutually independent. Notice also that Theorem \ref{kwiseindep} is usefull for 
values of $p$ that are close to $\frac{1}{2}$ and rather large values of $k$. 
As a direct application of Theorem \ref{Linial} one obtains the following, 
special case, of  Theorem \ref{SSS}. \\

\begin{thm} Fix $p\in (0,1)$ and 
let $X_1,\ldots,X_n$ be $k$-wise independent
Bernoulli $0/1$ random variables such that $\mathbb{E}\left[X_i\right]= p$, for each $i=1,\ldots,n$. 
Let $\varepsilon >0$ be such that $np(1+\varepsilon)$ is a positive integer that satisfies 
$np(1+\varepsilon) > k$. Then  
\[ \mathbb{P}\left[\sum_{i=1}^{n}X_i\geq np(1+\varepsilon)\right] \leq \binom{n}{k} p^{k} / \binom{np(1+\varepsilon)}{k} . \]
\end{thm}

The proof of this result is immediate and so is omitted. 
In the next subsection we shall be concerned with a particular 
dependency structure between Bernoulli 
$0/1$ random variables. 

\subsection{Dependency graphs}\label{lolole}

In this section we discuss yet another application of Theorem \ref{depHoeff}. 
We shall be concerned with sums of dependent Bernoulli random variables whose dependency 
structure is given in terms of a finite graph. 
Such a graph is referred to as a \emph{dependency graph} and 
is defined in the following Theorem.  Dependency graphs are used 
in probabilistic combinatorics in order to prove existence of "structures" with certain desired properties; 
a celebrated  tool for proving such existence is  the so-called Lov\'asz Local Lemma (see \cite{Erdos}).  
Below we obtain a concentration bound regarding sums of Bernoulli random variables 
whose dependency structure is given in terms of a finite graph. 
Recall that the \emph{independence number} of a finite graph is the cardinality  of the largest 
set of vertices no two of which are adjacent.  \\

\begin{thm}\label{dependencygraph} Let $G=(V,E)$ be a finite graph with vertices $v_1,\ldots,v_n$ and 
let $\alpha$ be its independence number. 
To each $v_i,i=1,\ldots,n$ 
we associate a Bernoulli $0/1$ random variable, $B_i$, such that $\mathbb{P}[B_i =1]=\frac{1}{2}$. 
Suppose that each random variable $B_i, i=1,\ldots,n$ is independent of the set $\{B_j : (v_i,v_j)\notin E\}$.
If $t$ is a real number such that $\frac{n}{2}<t<n$, then 
\[ \mathbb{P}\left[\sum_{i=1}^{n} B_i \geq t \right] \leq 2^{n-\alpha} H(n,1/2,t) , \]
where  $H(n,1/2,t)$ is the Hoeffding function, defined 
in Theorem \ref{folklore}. 
\end{thm}

Notice that the previous result reduces to Hoeffding's in case $\alpha =n$, i.e., the random variables are 
independent. Notice also the the result is useful for rather large values of $\alpha$.  
In the following section we discuss an improvement upon a concentration inequality for $U$-statistics. 

\subsection{$U$-statistics}\label{statu}

In this subsection we discuss an analogue of Hoeffding's Theorem \ref{folklore} 
for a particular class of weakly 
dependent random variables.  
Before being more precise, let us fix some notation. We will denote by 
$[n]^{d}_{<}$ the set consisting of all ordered $d$-tuples from the set $[n]$; formally, 
\[ [n]^{d}_{<} = \{(i_1,\ldots,i_d)\in [n]^d : 1\leq i_1 < i_2 < \cdots < i_d \leq n\} . \]
There are several instances in which one 
encounters sums of random variables of 
the form 
\[ \Xi:= \sum_{(i_1,\ldots,i_d)\in [n]^{d}_{<}} F(\xi_{i_1}, \xi_{i_2},\ldots,\xi_{i_d}) , \]
where $\xi_1,\ldots, \xi_n$ are \emph{independent and identically distributed} random variables and 
$F:\mathbb{R}^d\rightarrow [0,1]$ is a bounded function that depends only 
on the random vector $(\xi_{i_1},\xi_{i_2}\ldots,\xi_{i_d})$.  
Clearly, provided $d>1$, the random variable $\Xi$ is a sum of dependent random variables.
Such sums of random variables have been studied by several authors and are referred to as 
$U$-\emph{statistics}. 
Again, as an example of $U$-statistics, the reader may think of the number of 
triangles in an Erd\H{o}s-R\'enyi random graph, $G$, on $m$ 
vertices. In this case $d=3$ and $n=\binom{m}{3}$. Note that every triplet of vertices from $G$ 
uniquely determines a triplet of potential edges. Hence we can
set each $\xi_i, i=1,\ldots,n$ to be a Bernoulli random variable of parameter $p$ corresponding 
to the potential edges in $G$ and $F(\xi_{i_1},\xi_{i_2},\xi_{i_3})$ to be the indicator 
that the three potential edges, corresponding to a triplet of vertices, are all present in $G$ thus forming a triangle. \\
U-statistics is a class of unbiased estimators, introduced by Hoeffding \cite{HoeffdingUstat}, that 
has attracted considerable attention; see for example the works of Arcones \cite{Arcones}, Gin\'e et al. \cite{Gine}, Hoeffding 
\cite{HoeffdingUstat}, Janson \cite{Janson}, Joly et al. \cite{Joly}, just to name a few references. 
Let us bring to the reader's attention the following concentration inequality on $U$-statistics, 
which is due to 
Hoeffding  (see \cite{Hoeffdingone}, Section $5$; see also Janson \cite{Janson}, Section $4$). 
In order to avoid dealing with any rounding issues, we state the result for the case in which 
$d$ \emph{divides} $n$, i.e. $n=k\cdot d$, for some $k\in \{1,2,\ldots\}$.\\

\begin{thm}[Hoeffding, $1963$]
\label{ustatHoeff}
Let $d,n$ be positive integers such that $d$ divides $n$, 
i.e. $n=k\cdot d$, for some positive integer $k$. 
Suppose that $X$ is a random variable that can be written in the form 
\[ X = \sum_{(i_1,\ldots,i_d)\in [n]^{d}_{<}} F(\xi_{i_1},\ldots,\xi_{i_d}), \]
where $\xi_1,\ldots, \xi_n$ are \emph{independent and identically distributed} random variables and 
$F: \mathbb{R}^d \rightarrow [0,1]$ is a bounded function. Set $p:=\mathbb{E}[F(\xi_{i_1},\ldots, \xi_{i_d})], N_d :=\binom{n-1}{d-1}$. If $y= \mathbb{E}[X] + t\binom{n}{d}$, for 
some $t\in\left(0,1-p\right)$, then 
\[ \mathbb{P}\left[X \geq y \right]  \leq \inf_{h>0}\; e^{-hy}\; \mathbb{E}\left[e^{N_d\cdot B_{k,p}}\right] , \]
where $B_{k,p}$ is a binomial $\text{Bin}(k,p)$ random variable. Furthermore, 
\[ \inf_{h>0}\; e^{-hy}\; \mathbb{E}\left[e^{N_d\cdot B_{k,p}}\right]   \leq    \text{exp}\left(-2k t^2 \right)  . \]
\end{thm} 

We remark that a similar statement holds true for sums of independent random variables and 
has been the content of prior work (see \cite{PeRaWa}).  
By exploiting convexity, combined with 
similar ideas as in Section \ref{subsectionmartingale}, we deduce the 
following refined version of the previous theorem. \\

\begin{thm}\label{Ustatconcentration} 
Let $d,n$ be positive integers such that $d$ divides $n$, 
i.e. $n=k\cdot d$, for some positive integer $k$. 
Suppose that $X$ is a random variable that can be written in the form 
\[ X = \sum_{(i_1,\ldots,i_d)\in [n]^{d}_{<}} F(\xi_{i_1},\ldots,\xi_{i_d}), \]
where $\xi_1,\ldots, \xi_n$ are \emph{independent and identically distributed} random variables and 
$F: \mathbb{R}^d \rightarrow [0,1]$ is a bounded function. Set $p:=\mathbb{E}[F(\xi_{i_1},\ldots, \xi_{i_d})], N_d :=\binom{n-1}{d-1}$ and 
denote by 
$\mathcal{F}$ the set consisting of all functions $f:\mathbb{R} \rightarrow [0,+\infty)$ 
that are \emph{increasing} and \emph{convex}. If $y= \mathbb{E}[X] + t\binom{n}{d}$, for 
some $t\in\left(0,1-p\right)$, then 
\[ \mathbb{P}\left[X \geq y \right]  \leq \inf_{f\in \mathcal{F}}\; \frac{1}{f(y)} \mathbb{E}\left[f\left( N_d\cdot B_{k,p} \right)\right] , \]
where $B_{k,p}$ is a binomial $\text{Bin}(k,p)$ random variable. Moreover, if $t$ belongs to the interval
$\left(\frac{p(1-p)(e-1)}{1-p+ep},1-p\right)$  and $k(p+t)$ is a \emph{positive integer} from the interval $(kp, k)$, we have
\begin{eqnarray*} \inf_{f\in \mathcal{F}}\; \frac{1}{f(y)} \mathbb{E}\left[f\left( N_d\cdot B_{k,p} \right)\right]   &\leq&   \frac{hN_d +1}{e^{hN_d}}  \cdot\left(\text{exp}\left(-2k t^2 \right) - T_{2}(k,p,t)\right)\\
&+& \left(1- \frac{hN_d +1}{e^{hN_d}} \right)\mathbb{P}\left[B_{k,p}=k(p+t)\right] ,
\end{eqnarray*}
where 
\[  T_{2}(k,p,t):= \sum_{j=0}^{k(p+t)-1}e^{h(N_dj-y)}\mathbb{P}\left[B_{k,p}=j\right]   \]
and $h$ is the positive real satisfying 
\[ e^{-hy} \;
\mathbb{E}\left[e^{hB_{k,p}}\right]  = \inf_{s>0} \; e^{-(sy} \;
\mathbb{E}\left[e^{sB_{k,p}}\right] . \]
The later bound is strictly less than the bound of Theorem \ref{ustatHoeff}.
\end{thm} 

In other words, the previous result improves upon Theorem \ref{ustatHoeff} 
by adjusting a "missing factor" that is equal to $ \frac{hN_d +1}{e^{hN_d}} < 1$. 
We prove this result in 
Section \ref{ustatsection}. \\
The following five section are devoted to the proofs of the statements we discussed so far. 
Finally, in Section \ref{isolated}, we present some applications to the theory of random graphs.

\section{Covariance estimates}\label{mainres}

\subsection{Proofs of Theorems \ref{expfunct} and \ref{depHoeff}}

We  begin with the following Lemma in which we collect some properties of 
the random variables $Z_A$, defined in Theorem \ref{expfunct}. 
Recall that
$\partial_j[n]$ denotes  the family consisting 
of all subsets of $[n]$ of cardinality $j\in\{0,1,\ldots,n\}$.\\

\begin{lemma}\label{mainlem} Fix a positive integer $n$  and let $\{x_1,\ldots,x_n\}$ be real numbers from the interval $[0,1]$.
For every $A\subseteq [n]$ let $\zeta_A$ be defined as
\[ \zeta_A = \prod_{i\in A}x_i \prod_{i\in [n]\setminus A}(1-x_i) . \]
Then
\[ \sum_{A\subseteq [n]}\zeta_A = \sum_{j=0}^{n} \sum_{A \in \partial_j [n]} \zeta_A = 1 \quad \text{and}\quad \sum_{i=1}^{n} x_i =   \sum_{j=0}^{n}  j \sum_{A\in \partial_j[n]} \zeta_A  . \]
\end{lemma}
\begin{proof} The proofs of both statements are by induction on $n$.  
The first statement is clearly true for $n=1$. Assuming that it holds 
true for $n-1$, we prove it for $n>1$. 
Given a set $B\subseteq [n-1]$, we define $\zeta_B^{\prime} = \prod_{i\in B}x_i \prod_{i\in [n-1]\setminus B}(1-x_i)$.
Now notice that we can write 
\[ \sum_{j=0}^{n}\sum_{A \in \partial_j [n]} \zeta_A = \sum_{j=1}^{n}\sum_{n\in A \in \partial_j [n]} \zeta_A + \sum_{j=0}^{n-1}\sum_{n \notin A \in \partial_j [n]} \zeta_A ,
\]
where summation over $n\notin A\in \partial_j[n]$ means that the sum runs over those $A\in \partial_j[n]$ 
that do not contain $n$; similarly summation over $n \in A \in \partial_j [n]$ means 
that the sum runs over $A$ that contain $n$.
Now each term $\zeta_A$ in the first sum on the right hand side is multiplied by $x_n$ and each term in the 
second sum is multiplied by $1-x_n$. This implies that 
\[ \sum_{j=1}^{n}\sum_{n\in A \in \partial_j [n]} \zeta_A +\sum_{j=0}^{n-1} \sum_{n \notin A \in \partial_j [n]} \zeta_A =  x_n\sum_{j=0}^{n-1} \sum_{B \in \partial_j [n-1]} \zeta_B^{\prime} + (1-x_n)\sum_{j=0}^{n-1}\sum_{B \in \partial_j [n-1]} \zeta_B^{\prime} .\]
The inductional hypothesis finishes the proof of the first statement. The proof of the second statement is similar.
It is clearly true for $n=1$; assuming that it holds true for $n-1$ we prove it for $n>1$. 
Notice that  we can write
\[ \sum_{j=0}^{n} j  \sum_{A\in \partial_j[n]} \zeta_A  = \sum_{j=0}^{n-1} j \sum_{n\notin A\in \partial_j[n]} \zeta_A +  \sum_{j=1}^{n}j \sum_{n\in A\in \partial_j[n]} \zeta_A  .\]
The inductional hypothesis implies that the first addend in the right hand side of the last equation can be 
written as
 \[ \sum_{j=0}^{n-1} j \sum_{n\notin A\in \partial_j[n]} \zeta_A = (1-x_n)\sum_{j=0}^{n-1}j\sum_{B\in \partial_j[n-1]} \zeta_B^{\prime} = (1-x_n) \sum_{i=1}^{n-1}x_i . \]
 The second addend can be written as
 \begin{eqnarray*} 
 \sum_{j=1}^{n}j \sum_{n\in A\in \partial_j[n]} \zeta_A  &=&  x_n\sum_{j=0}^{n-1} (j+1)  \sum_{B\in \partial_j[n-1]}\zeta_B^{\prime}\\
&=&  x_n \sum_{j=0}^{n-1} j \sum_{B\in \partial_j[n-1]}\zeta_B^{\prime}   + x_n \sum_{j=0}^{n-1} \;  \sum_{B\in \partial_j[n-1]}\zeta_B^{\prime} \\
 &=& x_n \sum_{i=1}^{n-1}x_i + x_n , 
 \end{eqnarray*}
 where the last equality comes from the inductional hypothesis and the first statement. 
 Adding up the expressions in the last two equations yields the result.
\end{proof}

A basic ingredient in the proofs of most results in this paper   
is Theorem \ref{depHoeff}, which we are now in position to prove. 

\begin{proof}[Proof of Theorem \ref{depHoeff}] 
The claim that $\sum_A \mathbb{E}\left[Z_A\right] =1$, follows from the previous lemma.
Fix a function $f\in \mathcal{F}_{icx}$. Since $f$ is non-negative and increasing, 
Markov's inequality yields
\[ \mathbb{P}\left[\sum_{i=1}^{n}X_i \geq t\right] \leq  \frac{1}{f(t)}\; \mathbb{E}\left[f\left(\sum_{i=1}^{n}X_i\right)\right] . \]
Now Lemma \ref{mainlem} implies that
\[ \sum_{i=1}^{n} X_i =   \sum_{j=0}^{n}\;  \sum_{A\in \partial_j[n]} Z_A \cdot j  \]
and so, since $\sum_{A\subseteq [n]}Z_A =1$, it follows that 
$\sum_{i=1}^{n} X_i$ is a convex combination of the set of integers $\{0,1,\ldots,n\}$. 
Since $f$ is convex, we conclude 
\[ f\left(\sum_{i=1}^{n} X_i\right) \leq  \sum_{j=0}^{n}\;  \sum_{A\in \partial_j[n]} Z_A \cdot f(j) \] 
which, in turn, implies that
\[ \mathbb{E}\left[f\left(\sum_{i=1}^{n} X_i\right) \right]\leq  \sum_{j=0}^{n}\;  \sum_{A\in \partial_j[n]} \mathbb{E}\big[Z_A\big] \cdot f(j) .\]
The result follows.
\end{proof}

We can now proceed with the proof of Theorem \ref{expfunct}. 

\begin{proof}[Proof of Theorem \ref{expfunct}] 
Since $f(x)=e^{hx}, h,x>0$ is a convex, increasing and non-negative function, Theorem \ref{depHoeff}  
and the hypothesized estimate on $\mathbb{E}[Z_A]$ yield
\begin{eqnarray*} 
\mathbb{P}\left[\sum_{i=1}^{n}X_i \geq t\right] &\leq& e^{-ht}  \sum_{j=0}^{n}\;  \sum_{A\in \partial_j[n]} \mathbb{E}\big[Z_A\big] \cdot e^{hj} \\
&\leq& e^{-ht}  \sum_{j=0}^{n}\; \binom{n}{j} \gamma^{j} \delta^{n-j} \cdot e^{hj} \\
&=& e^{-ht} \left(\delta + \gamma e^{h} \right)^n ,\; \text{for}\; h>0,
\end{eqnarray*}
where the last equality follows from the binomial theorem. 
If we minimise the last expression with respect to $h>0$, we get 
$e^h = \frac{t\delta}{(n-t)\gamma}$. Therefore 
\[ \mathbb{P}\left[\sum_{i=1}^{n}X_i \geq t\right] \leq \gamma^{t} \delta^{n-t} \left(\frac{n-t}{t}\right)^t \left(\frac{n}{n-t}\right)^n   \]
and the first statement follows. 
Since $t=n\gamma + n\gamma\varepsilon$, 
we can write the right hand side of the last inequality as 
\[ \gamma^{t} \delta^{n-t} \left(\frac{n-t}{t}\right)^t \left(\frac{n}{n-t}\right)^n = 
\left\{ \left(\frac{\gamma}{\gamma + \gamma\varepsilon}\right)^{\gamma + \gamma\varepsilon} 
\left( \frac{\delta}{1-\gamma - \gamma\varepsilon} \right)^{1-\gamma - \gamma\varepsilon}   \right\}^n ,\]
which in turn is equal to
$\left(\frac{\gamma}{\gamma + \gamma\varepsilon}\right)^{\gamma + \gamma\varepsilon} 
\left( \frac{1-\gamma}{1-\gamma - \gamma\varepsilon} \right)^{1-\gamma - \gamma\varepsilon} 
\left( \frac{\delta}{1-\gamma} \right)^{1-\gamma - \gamma\varepsilon}$ and proves the result. 
\end{proof}

Likewise, Theorem \ref{Linial} can be obtained from Theorem \ref{depHoeff} by suitably 
choosing a function $f\in \mathcal{F}$. 

\begin{proof}[Proof of Theorem \ref{Linial}] 
We apply Theorem \ref{depHoeff} to a suitably chosen function.
Given positive integer $k$ such that $0<k<\beta n$, define the sequence $\{a_m\}_{m=0}^{n}$ 
by setting $a_m=0$, for $m\in\{0,1,\ldots,k-1\}$ and $a_m = \binom{m}{k}$, for $m\in \{k,k+1,\ldots,n\}$.  
Now let $g(x)$ 
to be the function defined by setting  $g(m)=a_m$, for $m\in\{0,1,\ldots,n\}$ 
and $g(\cdot)$ is linear between consecutive values, $g(m), g(m+1)$. 
It is easy to see, by comparing slopes, that the term $a_m, m\geq k$, is to the right of the 
line passing through the points 
$a_{m-1}$ and  $a_{m+1}$. This implies that 
$g(x)$ is convex, increasing and non-negative and so
Theorem \ref{depHoeff} yields
\[ \mathbb{P}\left[\sum_{i=1}^{n}X_i \geq \beta n\right] \leq \frac{1}{\binom{\beta n}{k}} \sum_{j=k}^{n} \sum_{A\in \partial_j[n]} \mathbb{E}\left[Z_A\right] \cdot g(j) = \frac{1}{\binom{\beta n}{k}} \sum_{A: |A|\geq k} \binom{|A|}{k} \mathbb{E}\left[Z_A\right]  . \]
Since the random variables $X_1,\ldots,X_n$ are indicators, 
the result follows upon observing that 
\[ \sum_{A: |A|\geq k} \binom{|A|}{k} \mathbb{E}\left[Z_A\right] = \sum_{A:|A|=k} \sum_{T: A\subseteq T} \mathbb{E}\left[Z_T\right] =
 \sum_{A:|A| =k} \mathbb{E}\left[\prod_{i\in A}X_i\right] . \]
\end{proof}

We proceed with yet another proof of the previous result and a corresponding lower bound.  

\begin{proof}[Proof of Theorem \ref{Linialrefined}] 
Given an outcome of the random variables $X_1,\ldots,X_n$, define $H_k$ to be 
the random variable that counts the number of indices $i\in\{1,\ldots,n\}$ for which
$X_i=1$ in $k$ draws without replacement from the set of indices $\{1,\ldots,n\}$. Notice that 
$H_k$ is a mixture of hypergeometric distributions. 
Now we can write 
\begin{eqnarray*} \mathbb{P}\left[H_k = k\right] &=& \sum_{j=k}^{n} \mathbb{P}\left[H_k=k\bigg| \sum_{i=1}^{n}X_i =j\right]\cdot \mathbb{P}\left[ \sum_{i=1}^{n}X_i =j\right] \\
&\geq& \sum_{j=\beta n}^{n} \mathbb{P}\left[H_k=k\bigg| \sum_{i=1}^{n}X_i =j\right]\cdot \mathbb{P}\left[ \sum_{i=1}^{n}X_i =j\right] \\
&=& \sum_{j=\beta n}^{n} \frac{\binom{j}{k}}{\binom{n}{k}} \cdot\mathbb{P}\left[ \sum_{i=1}^{n}X_i =j\right] \\
&\geq&  \frac{\binom{\beta n}{k}}{\binom{n}{k}} \cdot\mathbb{P}\left[ \sum_{i=1}^{n}X_i \geq \beta n\right].
\end{eqnarray*}
For $T\subseteq [n]$ set, as usual, $Z_T = \prod_{i\in T}X_i \prod_{i\in [n]\setminus T}(1-X_i)$. 
The upper bound follows upon observing that 
\begin{eqnarray*}\binom{n}{k} \mathbb{P}\left[H_k = k\right] &=&\sum_{j=k}^{n}\binom{j}{k} \cdot\mathbb{P}\left[ \sum_{i=1}^{n}X_i =j\right]
= \sum_{j=k}^{n}\binom{j}{k}  \sum_{T: |T|=j} \mathbb{E}\left[Z_T \right] \\
&=&  \sum_{A:|A|=k}\; \sum_{T: A\subseteq T} \mathbb{E}\left[Z_T\right] = \sum_{A: |A| =k} \mathbb{E}\left[\prod_{i\in A}X_i\right]. 
\end{eqnarray*}
We now proceed with the proof of the lower bound. 
Given an outcome of the random variables $X_1,\ldots,X_n$, define $H_{\beta n}$ to be 
the random variable that counts the number of indices $i\in\{1,\ldots,n\}$ for which
$X_i=1$ in $\beta n$ draws without replacement from the set of indices $\{1,\ldots,n\}$. 
Notice that 
\[ \mathbb{P}\left[ \sum_{i=1}^{n}X_i \geq \beta n\right] \geq \mathbb{P}\left[H_{\beta n} = \beta n\right] \]
and so it is enough to estimate $\mathbb{P}\left[H_{\beta n} = \beta n\right]$ from below. Now, using 
a similar computation as before, we can write
\begin{eqnarray*} \mathbb{P}\left[H_{\beta n} = \beta n\right] 
&=& \sum_{j=\beta n}^{n} \frac{\binom{j}{\beta n}}{\binom{n}{\beta n}} \cdot\mathbb{P}\left[ \sum_{i=1}^{n}X_i =j\right] \\
&=& \frac{1}{\binom{n}{\beta n}}  \sum_{j=\beta n}^{n} \binom{j}{\beta n} \sum_{T: |T|=j} \mathbb{E}\left[ Z_T \right] \\
&=& \frac{1}{\binom{n}{\beta n}} \sum_{A: |A| =\beta n} \mathbb{E}\left[\prod_{i\in A}X_i\right] .
\end{eqnarray*}
The result follows. 
\end{proof}

We end this section with the proof of Theorem \ref{bincoupling}. The proof will require the following, 
classical, result. \\

\begin{thm}[Hoeffding, $1956$]
\label{poissontrials} Let $B_1,\ldots,B_n$ be independent Bernoulli $0/1$ random 
variables whose mean equals $p_1,\ldots,p_n$, respectively. Then 
\[ \mathbb{P}\left[\sum_{i=1}^{n}B_i \geq b\right] \geq \mathbb{P}\left[B_{n,p}\geq b\right] , \]
when $0\leq b \leq np, \; p=\frac{1}{n}\sum_{i=1}^{n}p_i$ and $B_{n,p}$ is a binomial 
distribution of parameters $n$ and $p$.
\end{thm}
\begin{proof} See \cite{Hoeffding}, Theorem $4$.  
\end{proof}

The following proof is similar to the proof of Theorem $3.1$ form \cite{Impagliazzo}. See also Theorem $2$ in 
Siegel \cite{Siegelone}. 

\begin{proof}[Proof of Theorem \ref{bincoupling}] 
We may assume that $\mathbb{P}\left[\sum_i X_i \geq t\right] >0$.
For every outcome of the random variables 
$X_1,\ldots,X_n$, let $B_i$ be a Bernoulli $\text{Ber}(X_i), i=1,\ldots,n$, random variable. That is, 
given an outcome of $X_i, i=1,\ldots,n$, we toss $n$ independent $0/1$ coins such that the  
$i$-th coin lands on $1$ with probability $X_i$.   
Notice that, given $X_1,\ldots,X_n$, the mean of $\sum_{i=1}^{n}B_i$ equals $\sum_{i=1}^{n}X_i$ 
and so $\mathbb{E}\left[\sum_{i=1}^{n} B_i\right] = np$. 
Furthermore, given $X_1,\ldots,X_n$, define $B_n$ to be a binomial distribution of 
parameters $n$ and   
$\frac{1}{n}\sum_{i=1}^{n}X_i$; thus $\mathbb{E}[B_n]=np$ as well.
Since $\mathbb{P}\left[\sum_{i=1}^{n}X_i \geq t\right]>0$  we can write
\[ \mathbb{P}\left[\sum_{i=1}^{n}B_i \geq t-1 \bigg| \sum_{i=1}^{n}X_i \geq t \right] \leq 
\frac{\mathbb{P}\left[\sum_{i=1}^{n}B_i \geq t-1\right]}{\mathbb{P}\left[\sum_{i=1}^{n}X_i \geq t \right]}    \]
and so it is enough to estimate the probability on the left hand side from below. 
Now, 
given that $\sum_{i=1}^{n}X_i \geq t$,  the mean of $\sum_{i=1}^{n}B_i$ is 
greater than or equal to $t$ and so Theorem 
\ref{poissontrials} implies that 
\[ \mathbb{P}\left[\sum_{i=1}^{n}B_i \geq t-1 \bigg| \sum_{i=1}^{n}X_i \geq t \right] \geq \mathbb{P}\left[B_{n} \geq t-1 \bigg| \sum_{i=1}^{n}X_i \geq t \right]  .\]
It is well-known (see Kaas et al. \cite{Kaas}) that a median of a binomial distribution of parameters 
$n$ and $p$ is greater than or equal to $np -1$. 
This implies that, given $\sum_{i=1}^{n}X_i \geq t$, the probability 
that $B_{n}$ is greater than or equal to $t-1$ is at least $\frac{1}{2}$. 
Summarising, we have shown 
\[ \mathbb{P}\left[\sum_{i=1}^{n}X_i \geq t\right] \leq 2\cdot \mathbb{P}\left[\sum_{i=1}^{n}B_i \geq t-1\right] \] 
and, since we assume $t-1> np$,  
we may apply Hoeffding's Theorem \ref{folklore}  to $\mathbb{P}\left[\sum_{i=1}^{n}B_i \geq t-1\right]$ 
and conclude the result.  
\end{proof}

\subsection{Independent random variables - Proof of Theorem \ref{folklore}}

In this section we show that our main result can be seen as 
a generalisation of Hoeffding's theorem. In particular 
we provide, as a consequence of Theorem \ref{depHoeff}, a proof of Theorem \ref{folklore}.  
Notice that, in case the random variables $X_1,\ldots,X_n$ are \emph{independent},
we have
\begin{eqnarray*} 
\sum_{A\in \partial_j[n]}\mathbb{E}\left[\prod_{i\in A} X_i \prod_{i\in [n]\setminus A}(1-X_i) \right] &=& \sum_{A\in \partial_j[n]} \left(\prod_{i\in A} \mathbb{E}[X_i]\prod_{i\in [n]\setminus A}(1-\mathbb{E}[X_i])\right)\\
&=& \mathbb{P}\left[H(p_1,\ldots,p_n) =j\right], 
\end{eqnarray*}
where $p_i =\mathbb{E}[X_i], i=1,\ldots,n$ and $H(p_1,\ldots,p_n)$ is the random variable that 
counts the number of successes in $n$ independent trials where, for $i=1,\ldots,n$, 
the $i$-th trial has probability of success $p_i$.
The proof of Theorem \ref{folklore} will be based on the following, well-known, result. \\

\begin{thm}[Hoeffding, $1956$]
\label{successes} Let $H(p_1,\ldots,p_n)$ be the random variable defined above and set 
$p = \frac{1}{n}\sum_{i=1}^{n}p_i$. Let $f:\mathbb{R}\rightarrow \mathbb{R}$ be a convex function. 
Then 
\[ \mathbb{E}\left[f\left(H(p_1,\ldots,p_n)\right)\right] \leq \mathbb{E}\left[f\left(B_{n,p}\right)\right] , \]
where $B_{n,p}$ is a binomial random variable of parameters $n$ and $p$.
\end{thm}
\begin{proof} See \cite{Hoeffding}, Theorem $3$. 
\end{proof}
 
We can now provide yet another proof of Hoeffding's result. 

\begin{proof}[Proof of Theorem \ref{folklore}]
The proof is similar to the proof of Theorem \ref{expfunct}. 
We have already seen that
\[ \sum_{A\in \partial_j[n]}\mathbb{E}\left[Z_A \right]  = \mathbb{P}\left[H(p_1,\ldots,p_n) =j\right] .\]
Therefore, Theorems \ref{depHoeff} and \ref{successes} together with the binomial theorem yield
\begin{eqnarray*}
\mathbb{P}\left[\sum_{i=1}^{n}X_i \geq t\right] &\leq& e^{-ht}  \sum_{j=0}^{n}\;  \mathbb{P}\left[H(p_1,\ldots,p_n) =j\right] \cdot e^{hj} \\
&\leq& e^{-ht}  \sum_{j=0}^{n}\; \binom{n}{j} p^{j}(1-p)^j \cdot e^{hj} \\
&=& e^{-ht} \left(1-p + p e^{h} \right)^n , \; \text{for}\; h>0 .
\end{eqnarray*}
The result follows upon minimising the last expression with respect to $h>0$.
\end{proof} 

\section{Martingales - Proof of Theorem \ref{genmcdm}}\label{martsection}

In this section we  prove Theorem \ref{genmcdm}. Before doing so, we need to be 
able to estimate expectations of products of certain martingale difference 
sequances. This is the content of the following result.\\

\begin{lemma}
\label{productmartingale} 
Let $Y_1,\ldots,Y_n$ be a martingale difference sequence with $-p_i\leq Y_i\leq 1-p_i$, for $i=1,\ldots,n$ 
and suitable constants $p_i\in (0,1)$. Fix a subset $A\subseteq [n]$. Then
\[ \mathbb{E}\left[\prod_{i\in A}(Y_i + p_i) \prod_{i\in [n]\setminus A} (1-p_i-Y_i)\right] \leq \prod_{i\in A}p_i \prod_{i\in [n]\setminus A} (1-p_i) . \]
\end{lemma}
\begin{proof} The proof is by induction on $n$. For $n=1$ the statement is clearly true. Assuming that 
it holds true for $n-1$, we prove it for $n$. 
Let $A$ be a subset of $[n]$. There are two case to consider. Either $n\in A$ or $n\notin A$. 
In the first case, 
the tower property of conditional expectations yields
\begin{eqnarray*} \mathbb{E}\left[\prod_{i\in A}(Y_i + p_i) \prod_{i\in [n]\setminus A} (1-p_i-Y_i)\right]   =\\ \mathbb{E}\left[\mathbb{E}\left[(Y_n+p_n)\prod_{i\in A; i\neq n}(Y_i + p_i) \prod_{i\in [n-1]\setminus A} (1-p_i-Y_i)\bigg| \mathcal{A}_{n-1} \right]\right], 
\end{eqnarray*}
where $\mathcal{A}_{n-1}$ denotes the $\sigma$-algebra generated by the random variables $Y_1,\ldots,Y_{n-1}$. Since $Y_1,\ldots,Y_n$ is a martingale difference sequence it follows that 
the latter quantity equals 
\[ p_n \;\mathbb{E}\left[ \prod_{i\in A; i\neq n}(Y_i + p_i) \prod_{i\in [n-1]\setminus A} (1-p_i-Y_i)\right] \leq\prod_{i\in A}p_i \prod_{i\in [n]\setminus A} (1-p_i) , \] 
where the inequality follows from the inductional hypothesis. The second case is proven similarly and 
so is left to the reader. 
\end{proof}

We are now ready to prove the main result of this section.

\begin{proof}[Proof of Theorem \ref{genmcdm}]  
Fix $f(\cdot)\in \mathcal{F}$. 
Since $f(\cdot)$ is non-negative and increasing, Markov's inequality implies  
\begin{eqnarray*} \mathbb{P}\left[\sum_{i=1}^{n} Y_i \geq nt  \right] &=& \mathbb{P}\left[\sum_{i=1}^{n} Y_i +np \geq nt + np  \right] \\
&\leq& \frac{1}{f(nt+np)}  \mathbb{E}\left[f\left(\sum_{i=1}^{n}(Y_i +p_i) \right)\right].  
\end{eqnarray*}
Since $\sum_{i=1}^{n}(Y_i +p_i)$ is a sum of $[0,1]$-valued random variables 
and $f$ is convex, we apply Theorem 
\ref{depHoeff} and Lemma \ref{productmartingale} to conclude 
\begin{eqnarray*}  \mathbb{E}\left[f\left(\sum_{i=1}^{n}(Y_i +p_i) \right)\right] &\leq& \sum_{j=0}^{n} \sum_{A\in \partial_j[n]} \mathbb{E}\left[\prod_{i\in A}(Y_i+p_i) \prod_{i\in [n]\setminus A}(1-p_i-Y_i)\right] f(j) \\
&\leq& \sum_{j=0}^{n} \sum_{A\in \partial_j[n]} \prod_{i\in A}p_i \prod_{i\in [n]\setminus A}(1-p_i) f(j) \\
&\leq& \sum_{j=0}^{n} \binom{n}{j} p^j (1-p)^{n-j} f(j),
\end{eqnarray*} 
where the last inequality 
follows from Theorem \ref{successes}. The first statement follows. 
In order to prove the second statement, and for the sake of completeness, let us first 
prove McDiarmid's exponential bound. Define $f_h(x)=e^{hx}$, for $h>0$. 
Then we know from the first statement that 
\[ \mathbb{P}\left[\sum_{i=1}^{n} Y_i \geq nt  \right] \leq e^{-hn(t+p)}  \sum_{j=0}^{n} \binom{n}{j} p^j (1-p)^{n-j} e^{hj} = e^{-hn(t+p)} \left(1-p + p e^h\right)^n . \]
If we now minimise the last expression with respect to $h>0$, we get that $h$ must 
be such that $e^h = \frac{(t+p)(1-p)}{p(1-p-t)}$, and so
\[ \mathbb{P}\left[\sum_{i=1}^{n} Y_i \geq nt  \right] \leq H_m(n,p,t), \]
where $H_m(n,p,t)$ is the function defined in Theorem \ref{McDiar}. The bound 
$H_m(n,p,t)\leq \text{exp}\left(-2t^2n\right)$ follows by employing standard 
estimates on the Kullback-Leibler distance. 
We now proceed with the second statement of the Theorem. 
In order to simplify the notation, let us set $\ell := n(p+t)$; recall that we assume $\ell$ is a positive integer.
Consider the function 
$g_h(x) = \max\{0,h(x-\ell)+1\}$, for the particular value $h$ obtained by  minimising 
$e^{-(snt+snp)} \mathbb{E}\left[e^{sB_{n,p}}\right]$ with respect to $s>0$.
The first statement implies that 
\[ \mathbb{P}\left[\sum_{i=1}^{n} Y_i \geq nt  \right] \leq \sum_{j\geq \ell-\frac{1}{h}}  \binom{n}{j} p^j (1-p)^{n-j} \left(h(j-\ell) +1\right)= \mathbb{E}\left[g_h\left(B_{n,p}\right)\right]. \]
Now we can write
\begin{eqnarray*}  
H_m(n,p,t) - \mathbb{E}\left[g_h\left(B_{n,p}\right)\right] &=& \sum_{j< \ell -\frac{1}{h}}  \binom{n}{j} p^j (1-p)^{n-j} e^{h(j-\ell)}\\
&+& \sum_{j\geq \ell-\frac{1}{h}}  \binom{n}{j} p^j (1-p)^{n-j} \left\{e^{h(j-\ell)}-\left(h(j-\ell) +1\right)\right\}.
\end{eqnarray*}
Notice that the assumption $t>\frac{p(1-p)(e-1)}{1-p+ep}$ implies that $h>1$ and therefore 
$\ell-\frac{1}{h} \in (\ell-1,\ell)$. The 
assumption that $\ell$ is a positive integer, i.e. a possible value of $j$,  implies 
that, for $j=\ell$, the second term in the right hand side of the last equation evaluates to $0$. 
The last two observations imply that 
\begin{eqnarray*}
H_m(n,p,t) - \mathbb{E}\left[g_h\left(B_{n,p}\right)\right] &=&
\sum_{j=0}^{\ell-1}  \binom{n}{j} p^j (1-p)^{n-j} e^{h(j-\ell)} \\
&+&  \sum_{j= \ell+1}^{n}  \binom{n}{j} p^j (1-p)^{n-j} \left\{e^{h(j-\ell)}-\left(h(j-\ell) +1\right)\right\} .
\end{eqnarray*}
Now the fact that the function $\frac{1+hx}{e^{hx}}$ is decreasing for $x\geq 1$ implies
$1-\frac{1+hx}{e^{hx}}\geq 1 - \frac{1+h}{e^h}$ for $x\geq 1$ and so we can estimate 
\[ e^{h(j-\ell)}-\left(h(j-\ell) +1\right)   = e^{h(j-\ell)}\left(1-\frac{1+h(j-\ell)}{e^{h(j-\ell)}}\right) \geq  e^{h(j-\ell)}\left(1-\frac{1+h}{e^{h}}\right),  \]
for $j\geq \ell +1$. This implies  
\begin{eqnarray*} 
H_m(n,p,t) - \mathbb{E}\left[g_h\left(B_{n,p}\right)\right]
&\geq& \sum_{j\leq \ell-1}  \binom{n}{j} p^j (1-p)^{n-j} e^{h(j-\ell)} \\
&+& \left(1 - \frac{1+h}{e^h}\right)  \sum_{j\geq \ell+1}  \binom{n}{j} p^j (1-p)^{n-j} e^{h(j-\ell)}
\end{eqnarray*}
or, equivalently, that
\begin{eqnarray*} 
\mathbb{E}\left[g_h\left(B_{n,p}\right)\right] &\leq& \frac{1+h}{e^{h}}\left( H_m(n,p,t) - \sum_{j\leq \ell-1}e^{h(j-\ell)}\mathbb{P}\left[B_{n,p}=j\right] \right)\\
&+& \left(1- \frac{1+h}{e^h}\right) \mathbb{P}\left[B_{n,p} =\ell\right]
\end{eqnarray*}
and the second statement of Theorem \ref{genmcdm} follows. 
The third statement has been proven in Section \ref{subsectionmartingale} and so the result follows.
\end{proof}

In other words, the previous result improves upon 
McDiarmid's by adding a "missing factor" that is equal to $\frac{1+h}{e^{h}}  < 1$. 

\section{$k$-wise independence - Proof of Theorem \ref{kwiseindep}}

This section is devoted to the proof of Theorem \ref{kwiseindep}. 

\begin{proof}[Proof of Theorem \ref{kwiseindep}] 
Markov's inequality implies that 
\[ \mathbb{P}\left[\sum_{i=1}^{n}X_i \geq np(1+\varepsilon)\right] \leq e^{-hnp(1+\varepsilon)} \mathbb{E}\left[e^{h\sum_{i=1}^{n}X_i}\right] . \]
From Theorem \ref{depHoeff} we know that 
\[ \mathbb{E}\left[e^{h\sum_{i=1}^{n}X_i}\right] \leq \sum_{j=0}^{n} \sum_{A\in \partial_j[n]} \; \mathbb{E}\left[\prod_{i\in A}X_i \prod_{i\in [n]\setminus A}(1-X_i)\right]  e^{hj} := Q . \]
Fix a subset $A\subseteq [n]$ such that $|A|\leq k$. Let $B$ be any subset of $[n]\setminus A$ 
such that $|B|= k - |A|$. Since $A$ and $B$ are disjoint and the random variables are $[0,1]$-valued 
and $k$-wise independent it follows
\[ \mathbb{E}\left[\prod_{i\in A}X_i \prod_{i\in [n]\setminus A}(1-X_i)\right] \leq \mathbb{E}\left[\prod_{i\in A}X_i \prod_{i\in B}(1-X_i)\right] = p^{|A|}(1-p)^{k-|A|} . \] 
Now fix a subset $A\subseteq [n]$ such that $|A| > k$. Let $A^{\prime}$ be any subset of $A$ of cardinality $k$.
Then 
\[ \mathbb{E}\left[\prod_{i\in A}X_i \prod_{i\in [n]\setminus A}(1-X_i)\right] \leq \mathbb{E}\left[\prod_{i\in A^{\prime}}X_i\right] = p^{k} . \]
The last two estimates yield
\begin{eqnarray*} Q &\leq& \sum_{j=0}^{k}\binom{n}{j} p^j (1-p)^{k-j}e^{hj} + \sum_{j=k+1}^{n} \binom{n}{j} p^k e^{hj} \\
&\leq& \frac{1}{(1-p)^{n-k}}\sum_{j=0}^{k}\binom{n}{j} p^j (1-p)^{n-j}e^{hj}\\
&+& \frac{1}{(p(1-p))^{n-k}} \sum_{j=k+1}^{n} \binom{n}{j} p^j (1-p)^{n-j} e^{hj}  \\
&\leq& \frac{1}{(p(1-p))^{n-k}} \sum_{j=0}^{n} \binom{n}{j} p^j(1-p)^{n-j} e^{hj} \\
&=&  \frac{1}{(p(1-p))^{n-k}} \left(1-p+pe^h\right)^n  , 
\end{eqnarray*}
where the last equation follows from the binomial theorem. Summarising, we have shown 
\[ \mathbb{P}\left[\sum_{i=1}^{n}X_i \geq np(1+\varepsilon)\right] \leq  \frac{e^{-hnp(1+\varepsilon)}}{(p(1-p))^{n-k}}  \left(1-p+pe^h\right)^n \]
and the result follows upon minimising the last expression with respect to $h>0$.
\end{proof}

\section{Dependency graphs - Proof of Theorem \ref{dependencygraph}}

In this section we prove Theorem \ref{dependencygraph}.The proof  is similar to the proof of 
Theorem \ref{kwiseindep}.

\begin{proof}[Proof of Theorem \ref{dependencygraph}] 
From Theorem \ref{depHoeff} we can infer  that 
\[ \mathbb{P}\left[\sum_{i=1}^{n}B_i \geq t\right] \leq e^{-ht} \sum_{j=0}^{n}\sum_{A\in\partial_j[n]} \mathbb{E}\left[\prod_{i\in A}B_i\prod_{i\in [n]\setminus A}(1-B_i)\right] e^{hj} . \]
Fix a subset $V_{\alpha} \subseteq V$, of cardinality $\alpha$, 
such that 
no two vertices of $V_{\alpha}$
are adjacent and let $I_{\alpha} = \{i_1,\ldots, i_{\alpha}\}$ be the indices of the vertices that 
belong to $V_{\alpha}$.
For every $A \subseteq [n]$, let us denote $L_A = A \cap I_{\alpha}$ and 
$R_A =([n]\setminus A)\cap I_{\alpha}$. Then, for all $j\in\{0,1,\ldots,n\}$ and all $A\in \partial_j[n]$, we have
\[ \prod_{i\in A}B_i\prod_{i\in [n]\setminus A}(1-B_i) \leq \prod_{i\in L_A}B_i\prod_{i\in R_A}(1-B_i). \]
Since the random variables $B_{i_k}, k=1,\ldots,\alpha$ are mutually independent, we conclude 
\[\mathbb{E}\left[\prod_{i\in A}B_i\prod_{i\in [n]\setminus A}(1-B_i)\right] \leq \left(\frac{1}{2}\right)^{\alpha}\]
which, in turn, implies that 
\[ \mathbb{P}\left[\sum_{i=1}^{n}B_i \geq t\right] \leq \left(\frac{1}{2}\right)^{\alpha} e^{-ht}\left(1+e^h\right)^n . \]
The result follows upon minimising the last expression with respect to $h>0$. 
\end{proof}

\section{$U$-statistics - Proof of Theorem \ref{Ustatconcentration}}\label{ustatsection}

This section is devoted to the proof of Theorem \ref{Ustatconcentration}.
The proof  combines similar ideas as above 
together with an adaptation of the proof of Theorem $2.1$ from Janson \cite{Janson} 
(see also Hoeffding \cite{Hoeffdingone}, Section $5$). The main idea is to 
express $X$ as a weighted sum $\sum_j w_j Y_j$ in such a way that 
each random variable $Y_i$ is a sum of \emph{independent} random variables.   

\begin{proof}[Proof of Theorem \ref{Ustatconcentration}] 
Since $\xi_1,\ldots,\xi_n$ are independent and identically distributed, 
it follows that the random variables $F(\xi_{i_1},\ldots,\xi_{i_d})$, for $(i_1,\ldots,i_d)\in [n]_{<}^{d}$, 
are identically distributed. 
Let $p$ be the expected value of the 
random variable $F(\xi_{i_1},\ldots,\xi_{i_d})$.  
Recall that we assume that $d$ divides $n$, i.e. $n=k\cdot d$, for some positive integer $k$. 
Let $\mathcal{P}$ be the set of all partitions of $[n]$ into $k$ subsets of cardinality $d$. We first need to 
know the proportion of partitions $P\in \mathcal{P}$ that contain a fixed $d$-set. To this end, notice that 
by symmetry each of the $\binom{n}{d}$ choices for a $d$-set belongs to the same 
number, say $a$, of partitions in the class $\mathcal{P}$. Furthermore, each element from 
$\mathcal{P}$ contains $k$ sets of cardinality $d$. Therefore, 
\[ \binom{n}{d} \cdot a = |\mathcal{P}| \cdot k \; \Leftrightarrow \; a = \frac{|\mathcal{P}|}{\binom{n-1}{d-1}}, \]
where $|\mathcal{P}|$ denotes the cardinality of $\mathcal{P}$. 
This implies that we can write 
\begin{eqnarray*} X &=& \sum_{(i_1,\ldots,i_d)\in [n]^{d}_{<}} F(\xi_{i_1}\ldots,\xi_{i_d}) \\
&=& \frac{1}{a} \sum_{P\in \mathcal{P}}\; \sum_{(i_1,\ldots,i_d) \in P} F(\xi_{i_1},\ldots,\xi_{i_d}).
\end{eqnarray*}
Notice that each term $X_P:=\sum_{(i_1,\ldots,i_d) \in P} F(\xi_{i_1},\ldots,\xi_{i_d})$ is a sum 
of \emph{independent} and identically distributed, $\{0,1\}$-valued random variables, $F(\xi_{i_1},\ldots,\xi_{i_d}), \{i_1,\ldots,i_d\}\in P$, whose mean equals $p$ or, in other words, $X_P$ is a binomial 
random variable of parameters $k$ and $p$. 
Furthermore, notice that $\mathbb{E}[X]= kp\binom{n-1}{d-1}$. \\
Fix a function $f(\cdot)\in \mathcal{F}$ and set $y:= \mathbb{E}[X] + t\binom{n}{d}, N_d := \binom{n-1}{d-1}$. 
Markov's inequality and the assumption that $f$ is 
non-negative and increasing imply 
\[ \mathbb{P}\left[X \geq y \right] \leq \frac{1}{f(y)} \mathbb{E}\left[f(X)\right] .\] 
The assumption that $f(\cdot)$ is convex yields
\begin{eqnarray*} \mathbb{E}\left[f\left( X\right)\right] &=& \mathbb{E}\left[f\left( \frac{1}{|\mathcal{P}|} \sum_{P\in\mathcal{P}} \frac{|\mathcal{P}|}{a}X_P\right)\right]  \\
&\leq& \sum_{P\in \mathcal{P}}
\frac{1}{|\mathcal{P}|}  \mathbb{E}\left[f\left( N_d \cdot X_P\right)\right] \\
&=& \mathbb{E}\left[f\left( N_d \cdot B_{k,p}\right)\right] ,
\end{eqnarray*}
where the last equation follows from the fact that each $X_P$ is a binomial random variable of parameters $k$ and $p$.
The first statement follows. For the sake of completeness, we proceed by proving the exponential bound 
in Theorem \ref{ustatHoeff}. 
Let $h$ be a positive real, to be chosen later, and consider the function $f_h(x) = e^{hx}, x\in \mathbb{R}$.
Clearly, $f_h\in \mathcal{F}$ and the first statement together with the binomial 
theorem yield
\begin{eqnarray*} \mathbb{P}\left[X \geq y \right] &\leq& e^{-hy} \sum_{j=0}^{k} \binom{k}{j} p^j (1-p)^{k-j} e^{hN_d j} \\
&=& e^{-hy} \left(1-p + pe^{hN_d}\right)^k, \; \text{for}\; h>0 .
\end{eqnarray*}
If we minimise the last expression with respect to $h>0$, we get that $h$ must satisfy $e^{hN_d}= \frac{(p+t)(1-p)}{p (1-p-t)}$. Substituting this into the last expression and recalling that 
$y= kN_d(p + t)$ gives
\begin{eqnarray*} \mathbb{P}\left[X \geq y \right]  &\leq& \left\{\left(\frac{p}{p+t}\right)^{p+t} \left(\frac{1-p}{1-p-t}\right)^{1-p-t}\right\}^{k} \\
&=& e^{-k D(p+t|| p)}\\
&\leq& e^{-2kt^2} ,
\end{eqnarray*}
where the last inequality follows from the  
standard estimate $D(p+t||p)\geq 2t^2$ on the Kullback-Leibler distance. 
We now prove the second statement. 
Let $h$ be such that $e^{hN_d}= \frac{(p+t)(1-p)}{p(1-p-t)}$  and let $g_h(x), x\in\mathbb{R}$,
be the function defined by $g_h(x) = \max\{0, h(N_dx-y) +1\}$.  The first statement implies 
\[  \mathbb{P}\left[X \geq y \right] \leq \sum_{j\geq \frac{hy-1}{hN_d}} \binom{k}{j} p^j (1-p)^{k-j}\left( h(N_dj-y) +1\right)  = \mathbb{E}\left[g_h\left(N_d B_{k,p}\right) \right], \]
where $B_{k,p}$ is a binomial random variable of parameters $k$ and $p$. 
Let us denote $H_u(k,p,t) = \sum_{j=0}^{k} \binom{k}{j} p^j (1-p)^{k-j} e^{h(N_d j-y)}$.
Recall that $y=kN_d(p+t)$ and notice 
that we can write
\begin{eqnarray*} 
H_u(k,p,t) - \mathbb{E}\left[g_h\left(N_d B_{k,p}\right) \right]
= \sum_{j <k(p+t)- \frac{1}{hN_d}} \binom{k}{j} p^j (1-p)^{k-j} e^{h(N_d j-y)}\\
+ \sum_{j\geq k(p+t)- \frac{1}{hN_d}} \binom{k}{j} p^j (1-p)^{k-j} \left(  e^{h(N_d j-y)} - \left( h(N_dj-y) +1\right) \right) .
\end{eqnarray*}
Since $t> \frac{p(1-p)(e-1)}{1-p+ep}$ if follows that $hN_d >1$ and therefore $k(p+t)- \frac{1}{hN_d}$
belongs to the interval $(k(p+t)-1, k(p+t))$. As $k(p+t)$ is assumed to be an integer we can rewrite 
the last equation as 
\begin{eqnarray*} 
H_u(k,p,t) - \mathbb{E}\left[g_h\left(N_d B_{k,p}\right) \right]
= \sum_{j \leq k(p+t)- 1} \binom{k}{j} p^j (1-p)^{k-j} e^{h(N_d j-y)}\\
+ \sum_{j\geq k(p+t)} \binom{k}{j} p^j (1-p)^{k-j} \left(  e^{h(N_d j-y)} - \left( h(N_dj-y) +1\right) \right) .
\end{eqnarray*}
Notice that for $j=k(p+t)$ the second term in the right hand side evaluates to $0$.
Since the function $\frac{1+hx}{e^{hx}}$ is decreasing for $x\geq 1$ we can estimate,  for 
every potive integer $j$ such that $j\geq k(p+t)+1$,
\begin{eqnarray*} 
e^{h(N_d j-y)} - \left( h(N_dj-y) +1\right) &=&  \left(1-\frac{h(N_dj-y) +1}{e^{h(N_d j-y)}}\right)e^{h(N_d j-y)} \\
&\geq&  \left(1-\frac{hN_d +1}{e^{hN_d}}\right)e^{h(N_d j-y)}
\end{eqnarray*}
and this implies that 
\begin{eqnarray*}  H_u(k,p,t) - \mathbb{E}\left[g_h\left(N_d B_{k,p}\right)\right] &\geq& \sum_{j\leq k(p+t)-1} \binom{k}{j} p^j (1-p)^{k-j} e^{h(N_d j-y)}\\
&+&\left(1- \frac{hN_d +1}{e^{hN_d}}\right)\sum_{j\geq k(p+t)+1} \binom{k}{j} p^j (1-p)^{k-j}  e^{h(N_d j-y)}  
\end{eqnarray*}
or, equivalently, that 
\begin{eqnarray*}   \mathbb{E}\left[g_h\left(N_d B_{k,p}\right)\right]  &\leq& \frac{hN_d +1}{e^{hN_d}} H_u(k,p,t) + \left(1-\frac{hN_d +1}{e^{hN_d}} \right) \mathbb{P}\left[B_{k,p}=j\right] \\
&-&  \frac{hN_d +1}{e^{hN_d}}\sum_{j=0}^{k(p+t)-1}e^{h(N_dj-y)}\mathbb{P}\left[B_{k,p}=j\right] 
\end{eqnarray*} 
and the second statement follows. 
To prove the third stetement, note that the previous bound is 
\[ \leq  \frac{hN_d +1}{e^{hN_d}}  \cdot\text{exp}\left(-2k t^2 \right) + \left(1- \frac{hN_d +1}{e^{hN_d}} \right)\mathbb{P}\left[B_{k,p}=k(p+t)\right] := Q. \] 
Now,  Hoeffding's Theorem \ref{folklore} implies that
\[\mathbb{P}\left[B_{k,p}=k(p+t)\right] \leq\mathbb{P}\left[B_{k,p}\geq k(p+t)\right]\leq \text{exp}\left(-2k t^2 \right) \] 
and the third statement follows from the fact that $Q$ is a  convex combination of $\text{exp}\left(-2k t^2 \right)$ and 
$\mathbb{P}\left[B_{k,p}=k(p+t)\right]$.
\end{proof}

\section{Some applications}\label{isolated}

In this section we discuss some applications of Theorem \ref{Impaglia} and 
Theorem \ref{Linial} to the theory of random graphs.   
Recall (see \cite{Impagliazzo}, Theorem $3.1$) that in this case the constant $c$ 
of Theorem \ref{Impaglia} is qual to $1$.
We employ this result in order to obtain concentration inequalities for particular 
sums of weakly dependent indicators that are encountered in the theory of 
Erd\H{o}s-R\'enyi random graphs. 
Let us mention that we do \emph{not} intend to provide optimal concentration bounds;
our intention is to emphasize that Theorems \ref{Impaglia}, \ref{Linial}, combined with some elementary 
combinatorial result, 
yields certain concentration inequalities  
in a rather direct and simple 
manner. \\
Notice that, in order to apply Theorem \ref{Impaglia} to a specific problem, one has to determine 
the constant $\gamma$. 
In this section we find this constant for particular problems from the theory of
Erd\H{o}s-R\'enyi random graphs. Recall that such graphs, on $n$ vertices, are constructed by 
joining pairs of \emph{labelled} vertices with probability $p\in (0,1)$, independently of all other pairs.  
Let $G\in \mathcal{G}(n,p)$ be an Erd\H{o}s-R\'enyi random graph and denote by $I_{n,p}$ the number of isolated 
vertices in $G$. Recall that a vertex is called \emph{isolated} is its degree equals zero. 
Below we provide a concentration inequality for $I_{n,p}$. 
For shaper results on this problem we refer the reader to 
Ghosh et al. \cite{Ghosh}. \\

\begin{prop} Let $G\in \mathcal{G}(n,p)$ be an Erd\H{o}s-R\'enyi random graph. Let $I_{n,p}$ be the number of 
isolated vertices in $G$ and fix a real number $t$ such that $n(1-p)^{(n-1)/2} \leq t < n$ and write 
$t= n(1-p)^{(n-1)/2}(1+\varepsilon)$, for some $\varepsilon >0$.
Then 
\[ \mathbb{P}\left[I_{n,p} \geq t\right] \leq  e^{-nD(\gamma(1+\varepsilon) || \gamma)},\]
where $\gamma = (1-p)^{(n-1)/2}$.
\end{prop}
\begin{proof} 
For every vertex $v_i,i=1,\ldots,n$, let $I_i$ be the indicator of the event "$v_i$ is isolated". 
Then $I_{n,p} = \sum_i I_i$ and $\mathbb{E}\left[I_{n,p}\right] = n(1-p)^{n-1}$. 
Let $A\subseteq [n]$ be a set of cardinality $j\in\{0,1,\ldots,n\}$. 
Theorem \ref{Impaglia} requires to find constant $\gamma$ such that 
$\mathbb{E}\left[\prod_{i\in A}I_i\right] \leq \gamma^j$.
We claim that we may choose $\gamma = (1-p)^{(n-1)/2}$; the result then 
follows from Theorem \ref{Impaglia}. To prove the claim, notice that the expression on the left hand side 
of the previous inequality equals the probability that the vertices $v_i, i\in A$ are isolated, 
which happens with probability $(1-p)^{\binom{j}{2}+j(n-j)}$. Hence 
\begin{eqnarray*} \mathbb{E}\left[\prod_{i\in A}I_i\right] = (1-p)^{\binom{j}{2}+j(n-j)} 
\leq (1-p)^{\frac{(n-1)j}{2}}  ,
\end{eqnarray*}
as required. 
\end{proof}

We now proceed with a concentration inequality on yet another sum of dependent indicators. 
Let $G\in \mathcal{G}(n,p)$ be an Erd\H{o}s-R\'enyi random graph and denote by $T_{n,p}$ the number of 
triangles in $G$. The problem of obtaining upper bounds on the probability that $T_{n,p}$ is larger 
than its mean is classical; we refer the reader to the works of DeMarco et al. \cite{Demarco}, 
Janson \cite{Janson} and Kim et al. 
\cite{KimVu} for much sharper bounds and references. Using Theorem \ref{Impaglia} one can 
obtain the following concentration bound on the number of triangles in a random graph. \\

\begin{prop}\label{trigono} Let $G\in \mathcal{G}(n,p)$ be an Erd\H{o}s-R\'enyi random graph 
and denote by $T_{n,p}$ the number of 
triangles in $G$. Fix a real number $t$ such that $\binom{n}{3} p^{\frac{3}{n-2}} \leq t < \binom{n}{3}$ and 
write $t = \binom{n}{3} p^{\frac{3}{n-2}}(1+\varepsilon)$, for some $\varepsilon >0$.
Then 
\[ \mathbb{P}\left[T_{n,p}\geq t\right] \leq e^{-\binom{n}{3}D(\gamma(1+\varepsilon) || \gamma)},\]
where $\gamma = p^{\frac{3}{n-2}}$.
\end{prop}

The proof of the previous result   
is based upon the following Mantel-type result. \\

\begin{lemma}\label{extrelem} 
Fix $n\geq 3$, set $N=\binom{n}{3}$ and
suppose that $G=(V,E)$ is a graph on $n$ vertices having $j$ triangles, where $j\in \{0,1,\ldots,N\}$. 
For every triangle $T_i,i=1,\ldots,j$ in $G$, let $E_i$ be the set consisting of the three edges that 
belong to $T_i$ and set $R= \cup_{i=1}^{j} E_i$.
Then $R$ contains at least $\frac{3j}{n-2}$ edges. 
\end{lemma}
\begin{proof} Let $|R|$ denote the cardinality of $R$. 
Count pairs $(e,v)$, where $e=(v_1,v_2)$ is an edge from $R$ and $v$ is a vertex from $V\setminus \{v_1,v_2\}$. Now, on one hand, 
the number of such pairs is at most $|E|\cdot(n-2)$. On the other hand, each triangle of $G$  
is counted exactly three times. Thus 
\[ |R|\cdot(n-2) \geq 3j \]
and the result follows. 
\end{proof}

\begin{proof}[Proof of Proposition \ref{trigono}] 
Set $N:= \binom{n}{3}$.
Let $T_i, i=1,\ldots, N$ be an enumeration of all potential triangles in $G$.  
Given a triangle $T_i$, let $E_i$ denote the set consisting of the three edges that belong to $T_i, i=1,\ldots,N$.
Define $X_i$ to be the indicator of the event that triangle $T_i$ is present in $G$.
Then the number of triangles in $G$ equals $T_G=\sum_{i} X_i$.
In order to apply Theorem \ref{Impaglia}
we need to find an upper bound on 
$\mathbb{E}\left[\prod_{i \in A} X_i\right]$, for $A \in \partial_j[N]$ and $j\in \{0,\ldots,N\}$.
Let $j\in \{0,1,\ldots,N\}$ be such that there exists graphs on $n$ vertices having $j$ triangles and note that, 
for $A\in \partial_j[N]$, we have 
\[ \mathbb{E}[Z_A] = \mathbb{P}\left[X_i=1,\; \text{for}\; i\in A\right] = p^{\ell}   , \]
where $\ell$ denotes the cardinality of the set $\cup_{i\in A} E_i$. 
Lemma \ref{extrelem} finishes the proof.
\end{proof}

Clearly, Proposition \ref{trigono} is not very informative; the constant $\gamma$ is quire large. 
Perhaps more sophisticated versions of Lemma \ref{extrelem} can provide smaller values of $\gamma$. 
Let us remark that Lemma \ref{extrelem} may be iterated to produce bounds on the number of cliques 
in an Erd\H{o}s-R\'enyi random graph. Let us illustrate this with the number of $4$-cliques. 
Let $G\in \mathcal{G}(n,p)$ be an Erd\H{o}s-R\'enyi random graph and denote by $Q_{n,p}$ the number of 
$4$-cliques in $G$. We first provide a lower bound on $4$-cliques in a graph in terms of triangles. \\

\begin{lemma}\label{extrelemtwo} 
Fix $n\geq 4$, set $N_4=\binom{n}{4}$ and
suppose that $G=(V,E)$ is a graph on $n$ vertices having $k$ $4$-cliques, where $k\in \{0,1,\ldots,N_4\}$. 
For every $4$-clique $Q_i,i=1,\ldots,j$ in $G$, let $T_i$ be the set consisting of the four triangles that 
belong to $Q_i$ and set $R= \cup_{i=1}^{j} T_i$.
Then $R$ contains at least $\frac{4k}{n-3}$ triangles and so, by Lemma \ref{extrelem}, at least 
$\frac{12 k}{(n-2)(n-3)}$ edges.  
\end{lemma}
\begin{proof} 
Count pairs $(T,v)$, where $T$ is a triangle from $R$ and $v$ is a vertex from $V$ which 
is different from the vertices of the triangle $T$.  The number 
of such pairs is at most $|R|\cdot (n-3)$ and each $4$-clique is counted exactly $4$ times. 
\end{proof}

We can therefore conclude the following, rather crude, bound whose proof is similar 
to the proof of Proposition \ref{trigono} and so is left to the reader. \\

\begin{prop} Let $G\in \mathcal{G}(n,p)$ be an Erd\H{o}s-R\'enyi random graph 
and denote by $Q_{n,p}$ the number of 
$4$-cliques in $G$. Fix a real number $t$ such that $\binom{n}{4} p^{\frac{12}{(n-2)(n-3)}} \leq t < \binom{n}{3}$ and 
write $t = \binom{n}{4} p^{\frac{12}{(n-2)(n-3)}}(1+\varepsilon)$, for some $\varepsilon >0$.
Then 
\[ \mathbb{P}\left[Q_{n,p}\geq t\right] \leq e^{-\binom{n}{4}D(\gamma(1+\varepsilon) || \gamma)},\]
where $\gamma = p^{\frac{12}{(n-2)(n-3)}}$.
\end{prop}

Let us proceed with some applications of Theorem \ref{Linial} to another model of random graphs, 
namely $G(n,m)$. Recall that such a graph is obtained by selecting uniformly at random 
a graph, $G$, from the set of all labelled graphs on $n$ vertices and $m$ edges. 
We begin with a concentration bound on the number of isolated vertices in $G$. \\

\begin{prop} Let $G\in G(n,m)$ and denote by $I_{n,m}$ the number of isolated vertices in $G$. 
Let $t$ be a positive integer. Then
\[ \mathbb{P}\left[I_{n,m}\geq t\right] \leq \min_{0<k<t} \binom{n}{k}  \binom{\binom{n-k}{2}}{m}/\binom{t}{k}\binom{\binom{n}{2}}{m}  . \]
\end{prop}
\begin{proof} 
Let $I_i$ be the indicator of the event "vertex $i$ is isolated". Then $I_{n,m}=\sum_i I_i$.
For fixed $A\subseteq [n]$ of cardinality $k$ we have
$\mathbb{E}\left[\prod_{i\in A}I_i\right]= \mathbb{P}\left[I_i =1, \text{for all}\; i\in A\right]$ 
and the later probability equals  
\[ \mathbb{P}\left[I_i =1, \text{for}\; i\in A\right] = \binom{\binom{n-k}{2}}{m}/\binom{\binom{n}{2}}{m} .\]
Theorem \ref{Linial} finishes the proof. 
\end{proof}

Off course, the previous bound is useful for $t$ such that 
$t> n\binom{n-1}{m}/\binom{n}{m}$. Note that, for such $t$, 
the bound of the previous result reduces to Markov's inequality when $k=1$ and so the minimum over 
$k\in \{1,\ldots, t-1\}$ provides a better bound than Markov's inequality. 
Our paper ends with a concentration bound on the number of triangles in $G\in G(n,m)$. \\

\begin{prop} Let $G\in G(n,m)$ and denote by $T_{n,m}$ the number of triangles in $G$. 
Let $t$ be a positive integer from the set $\{2,\ldots, \binom{n}{3}\}$. Then
\[ \mathbb{P}\left[T_{n,m}\geq t\right] \leq \min_{0<k<t} \binom{\binom{n}{3}}{k}  \binom{\binom{n}{2}-\lfloor\frac{3k}{n-2}\rfloor}{m-\lfloor\frac{3k}{n-2}\rfloor}/\binom{t}{k}\binom{\binom{n}{2}}{m}  . \]
\end{prop}
\begin{proof} 
Set $N=\binom{n}{3}$. 
Let $T_i,i=1,\ldots, N$ be an enumeration of all potential triangles, let $V_i$ be set consisting 
of the three vertices of triangle $T_i$ and let $E_i$ be the set consisting of the three edges of $T_i$.
Let $I_i$ be the indicator of the event "triangle $i$ is present in $G$". Then $T_{n,m}=\sum_i I_i$.
Fix positive integer $k$ such that $0<k<t$. 
If $A\subseteq [N]$ is a set of indices of cardinality $k$ then 
$\mathbb{E}\left[\prod_{i\in A}I_i\right]$ equals the probability 
that the triangles $T_i,i\in A$ are all present in $G$. 
The set of vertices $\cup_i V_i$ and the set of edges $\cup_i E_i$ induce a (potential) graph 
that has $k$ triangles and so, 
by Lemma  \ref{extrelem}, it has at least $\frac{3k}{|\cup_i V_i|-2}\geq \frac{3k}{n-2}$ edges. We can thus 
associate to each, non-empty, collection $\{T_i\}_i$ of $k$ triangles a set, $E^{\prime}$, 
consisting of $\lfloor\frac{3k}{n-2}\rfloor$ edges in such a way that 
if the triangles $\{T_i\}_i$ are present in $G$ then the edges from $E^{\prime}$ are also present in $G$.
This implies that 
\[ \mathbb{P}\left[I_i =1, \text{for}\; i\in A\right] \leq  \binom{\binom{n}{2}-\lfloor\frac{3k}{n-2}\rfloor}{m-\lfloor\frac{3k}{n-2}\rfloor}/\binom{\binom{n}{2}}{m} \]
and the result follows from Theorem \ref{Linial}. 
\end{proof}

\textbf{Acknowledgements}
The authors are supported by ERC Starting Grant 240186 "MiGraNT, Mining Graphs and Networks: a Theory-based approach". We are grateful to Dr. Yuyi Wang for fruitful discussions and valuable comments.

\end{document}